\documentclass[11pt, letterpaper, reqno]{amsart}
\usepackage[utf8]{inputenc}
\usepackage{amsmath}
\usepackage{amsfonts}
\usepackage{amssymb}
\usepackage{amsthm}
\usepackage{dsfont}

\usepackage[backref=page,breaklinks]{hyperref}
\usepackage{cleveref}

\usepackage{breakurl}
\usepackage{stmaryrd}
\usepackage{mathdots}
\usepackage{tikz-cd}
\usepackage{float}
\usepackage[linesnumbered,lined,boxed,commentsnumbered,ruled,algosection]{algorithm2e}
\usepackage[shortlabels, inline]{enumitem}
\usepackage{colonequals}

\DeclareMathOperator{\End}{End}

\DeclareMathOperator{\lspan}{span}
\DeclareMathOperator{\Mat}{Mat}
\DeclareMathOperator{\Jac}{Jac}
\DeclareMathOperator{\Prym}{Prym}
\DeclareMathOperator{\HH}{H}

\DeclareMathOperator{\Sp}{Sp}

\DeclareMathOperator{\id}{id}
\DeclareMathOperator{\im}{im}

\DeclareMathOperator{\diag}{diag}

\DeclareMathOperator{\CC}{\mathbb{C}}
\DeclareMathOperator{\QQ}{\mathbb{Q}}

\DeclareMathOperator{\ZZ}{\mathbb{Z}}

\newcommand{\F}{\mathbb{F}}
\newcommand{\PP}{\mathbb{P}}
\newcommand{\Magma}{\textsf{Magma}}

\newcommand{\iu}{\mathrm{i}\mkern1mu}
\newcommand{\thetacha}[3]{\vartheta_{#1}\hspace{-0.5mm}\big[\hspace{-0.5mm}\begin{smallmatrix}
#2\\#3
\end{smallmatrix}\hspace{-0.5mm} \big]}

\theoremstyle{definition}
\newtheorem{definition}[algocf]{Definition}
\newtheorem{example}[algocf]{Example}
\newtheorem{rem}[algocf]{Remark}

\newtheorem{overview}[algocf]{Overview}
\theoremstyle{plain}
\newtheorem{thm}[algocf]{Theorem}
\newtheorem*{thm*}{Theorem}

\newtheorem{lem}[algocf]{Lemma}
\newtheorem{prop}[algocf]{Proposition}

\numberwithin{equation}{section}

\begin{document}
\title[Equations of genus $4$ curves from their theta constants]{Equations of genus $4$ curves from their theta constants}
\begin{abstract}
In this article we give explicit formulas for the equations of a generic genus $4$ curve in terms of its theta constants. The method uses the Prym construction and the beautiful classical geometry around it.
\end{abstract}
\author[Hanselman, Pieper, Schiavone]{Jeroen Hanselman, Andreas Pieper, Sam Schiavone}
\address{
  Jeroen Hanselman,
  RPTU Kaiserslautern-Landau
}
\email{hanselman@mathematik.uni-kl.de}
\address{
  Andreas Pieper,
  Universit\"at Duisburg-Essen
}
\email{andreas.pieper@uni-due.de}
\address{
  Sam Schiavone
}
\email{sam.schiavone@gmail.com}

\maketitle
\section{Introduction}

\par The goal of this article is to provide closed formulas, in the spirit of Aronhold--Weber,  for recovering explicit equations of a generic genus 4 curve $C$, given the values of its algebraic theta constants. The outcome is a quadric $Q$ and a cubic $\Gamma$ in $\mathbb{P}^3$ such that $Q\cap \Gamma$ is equal to the image of the canonical embedding of $C$. Our main result is the following (Theorem \ref{thm:main_result} below).

\begin{thm*} Let $C$ be a generic genus $4$ curve over an algebraically closed field $k$ of characteristic 0 or of characteristic $p$ with $p$ large enough. There are explicit formulas for the equations of $C$ in terms of its theta constants. They use only the following algebraic operations: elementary arithmetic, taking square roots, and solving linear systems of equations.
\end{thm*}
We record the explicit formulas for these equations in \Cref{subsec:main}.

\subsection{Theta constants}
Recall that over $k=\mathbb{C}$ the theta constants of a principally polarized abelian variety (p.p.a.v.) of dimension $g$ with (small) period matrix $\tau\in \mathfrak{H}_g$ are given by the $2^{g-1}(2^g+1)$ numbers $\thetacha{}{\delta}{\varepsilon}(0, \tau)$ with $\delta,\varepsilon\in \frac{1}{2}\!\ZZ^g $ running through a set of representatives of pairs in $\frac{1}{2}\!\ZZ^g\!/\!\ZZ^g$ satisfying $4\delta^t \varepsilon\equiv 0 \mod 2$. Here $\thetacha{}{\delta}{\varepsilon}$ is the Riemann theta function with characteristics, given by
$$
\thetacha{}{\delta}{\varepsilon} (z, \tau)= \sum_{n\in {\ZZ}^g} \exp\left( \pi \iu (n+\delta)^t \tau (n+\delta) +2 \pi \iu (n+\delta)^t(z+\varepsilon)\right)\,.
$$

Mumford \cite{MumfordEqAV1} gave a purely algebraic definition of the theta constants that works over any algebraically closed field $k$ of characteristic $\neq 2$. He also showed that the theta constants determine the principally polarized abelian variety uniquely. 

It is then natural to ask if one can express the equations of a curve $C$ in terms of the theta constants of its Jacobian. This problem was classically studied by Rosenhain (genus $2$) and Aronhold--Weber (plane quartics). Later Takase \cite{Takase} generalized Rosenhain’s work to arbitrary hyperelliptic curves. In all these cases they exploit the fact that the moduli space of these curves equipped with a full level $2$ structure is rational. In the hyperelliptic case this rationality follows from the existence of Weierstrass equations, and for plane quartics the formulas of Aronhold (see \Cref{subsec:aronholdweber}) give an explicit rational parametrization.
\par In genus 4 we are faced with a different situation, as the moduli space of genus $4$ curves with a full level $2$ structure seems to not be unirational. (This will be shown in future work.) This fact may explain why the general problem has remained open, only garnering results in special cases. For instance, Schottky \cite{SchottkySpecial} managed to solve the problem in the case where one of the theta constants vanishes, which is equivalent to the canonical quadric being a cone.
\par Another phenomenon not appearing in the theory of plane quartics is that not every p.p.a.v.~is a Jacobian when $g\geqslant 4$. In fact, a $4$-dimensional (indecomposable) p.p.a.v. is the Jacobian of a smooth curve if and only if the Schottky modular form vanishes \cite[Chapter 5]{FarkasRauch}\cite{IgusaSchottky}\cite{Freitag}.
\par Recall from \cite[Chapter 5]{FarkasRauch} that the Schottky modular form is related to the Prym construction as follows. Given an \'etale double cover $\tilde{C} \to C$, the Prym variety $\Prym(\tilde{C}/C)=\ker(\text{Nm}:\Jac(\tilde{C}) \rightarrow \Jac(C))^\circ$ is in a natural way a principally polarized abelian variety. In particular, the theta constants of $\Prym(\tilde{C}/C)$ must satisfy the quartic Riemann theta formulas. The Schottky modular form is derived from these identities by expressing the theta constants of $\Prym(\tilde{C}/C)$ in terms of the theta constants of $\Jac(C)$.
\par This suggests that the Prym construction could play a role in the problem of reconstructing the curve from the theta constants, and this is the approach we pursue here. Furthermore, we will use a beautiful classical geometric construction concerning the Prym due to Caporali, Wirtinger, P. Roth (see \cite{coble} for a classical textbook reference), and W. P. Milne \cite{Milne}. The modern development was pioneered by Recillas \cite{RecillasPhD} and henceforth it is commonly referred to as Recillas's trigonal construction. For further references, see \cite{Catanese}\cite{DonagiTetrag}\cite{Lehavi}\cite{BruinSertoz}.

\subsection{Motivation and applications}
In addition to the intrinsic value of recovering a genus 4 curve from its theta constants, our formulas also have several applications. It is interesting in arithmetic and complex geometry to find higher genus curves with special properties, such as endomorphism rings, Hodge tensors, and Galois representations. Being able to reconstruct a curve from the theta constants (or period matrix) of its Jacobian allows us to construct curves having these properties. In the present article we will discuss two examples: one curve whose Jacobian is isogenous to a product, and one curve with a modular Jacobian (see \Cref{sec:examples}). In the recent preprint \cite{BHPS} the authors together with Bouchet found two special families of genus 4 curves of Mumford type using our formulas in combination with the invariant theoretic techniques of \cite{ThomasInvariants, ThomasReconstructing}. These provide the first examples of Shimura families not of PEL-type.

\subsection{Previous work}
The task of recovering a curve from its Jacobian is an important aspect of the Schottky problem, and has seen three main approaches (see \cite{Grushevsky} for an overview):
\begin{enumerate}[(i)]
    \item The various methods based on Fay's trisecant identity \cite[Formula (45)]{Faytrisec}\cite[Section IIIb]{TataII};
    \item The Andreotti--Mayer approach, using singularities of the $\Theta$-divisor; and
    \item The Schottky--Jung approach, using Prym varieties.
\end{enumerate}
The first approach has been the most fruitful so far, with the most significant landmarks being the characterization of Jacobians via the KP-equations \cite{Shiota}, or via the existence of a trisecant of the Kummer variety \cite{Krichever}. Furthermore, in the context of the KP-equations, Dubrovin \cite{Dubrovin} presented a method for recovering a curve from its Jacobian using only theta functions. Agostini, \c{C}elik, and Eken \cite{agostini} have implemented this to numerically reconstruct a curve of arbitrary genus from the period matrix of their Jacobian. They compute theta derivatives up to order $4$ to produce degree $4$ polynomials vanishing on the curve. In genus $4$ they require Gr\"obner basis computations to write the equations as the intersection of a quadric and a cubic.
\par The Andreotti--Mayer approach to the Schottky problem works very well in genus four as Jacobians are characterized by the $\Theta$-divisor being singular. In the non-hyperelliptic case this singular point is unique and Kempf \cite{Kempf} describes a method to reconstruct the curve from the tangent cone of this singularity. This was implemented in the work of Chua, Kummer, and Sturmfels \cite{Chua_2018}. In order to find this singular point they need to solve a non-linear system for computing the point where both a theta function and all its derivatives vanish.

The present paper pursues the Schottky--Jung approach and is strongly influenced by the ideas of Lehavi \cite{Lehavi}. He makes a result of Carporaso--Sernesi effective in genus $4$, who have proven that a general curve of genus $g$ can be recovered from its odd theta hyperplanes \cite{CaporasoSernesi}.

Lehavi's method is based on the cartesian diagrams
\begin{equation} 
\begin{tikzcd}
    S^2 \HH^0(\Omega_C(\eta))\times_{\HH^0(\Omega_C^{\otimes 2})} S^2 \HH^0(\Omega_C)\arrow{r}\arrow[hookrightarrow]{d}\arrow{r}{\varphi } & S^2 \HH^0(\Omega_C(\eta))\arrow[hookrightarrow]{d}\\
     S^2 \HH^0(\Omega_C)\arrow{r} & \HH^0(\Omega_C^{\otimes 2}) 
\end{tikzcd}
\end{equation}
where $\eta$ runs through all the 255 non-trivial two-torsion points in $\Jac(C)[2]~\setminus~\{0\}$. The quadric containing $C$ is recovered from these 255 cartesian squares using only linear algebra. However, to recover the cubic, one must find the locus where a certain  $3\times 4$-matrix with linear entries has rank one. Lehavi's paper contains an error, and if one tries to fix it, one is required to solve this non-linear system for all the 255 non-zero two torsion points \cite[\S 5.2]{CelikKulkarni}. Consequently, Lehavi's results, despite being geometrically elegant, do not immediately provide an efficient algebraic way of recovering the curve.

The present paper improves on this method in several respects: we use only a single two-torsion point $\eta$, and we require no non-linear algebra to recover the cubic. Furthermore, we derive the first closed formula for the equations of $C$ in terms of the theta constants. In contrast to the papers mentioned above, we only use basic algebraic operations, namely elementary arithmetic, extraction of square roots, and solving linear systems of equations to compute equations for the canonical model of the curve.

The formula can be used to perform computations over complex fields with high precision. For example by computing the theta constants by first employing the duplication formula to reduce to calculating the $2^g=16$ fundamental theta constants $\thetacha{}{\delta}{0}$, and then using the work of Elkies and Kieffer \cite{ElkiesKieffer}. Consequently, it provides access to many examples arising in number theory that were previously beyond the reach of existing methods.

\subsection{Structure of the article} 
The article is organized as follows. In Section \ref{sec:prelims} we begin by defining the objects that will play key roles in our construction. In Section \ref{sec:main} we describe our main result: formulas for the equations of a genus 4 curve in terms of the theta constants of its Jacobian. In Section \ref{subsec:main} we summarize and explicitly write down all the formulas and equations involved. In Section \ref{sec:examples} we present an explicit computation over a finite field, and two examples as an application of our method. Finally, in Section \ref{sec:future} we briefly discuss how our method can be generalized to higher genus.

\subsection{Acknowledgements}
The authors would like to thank Thomas Bouchet, Nils Bruin, Edgar Costa, Igor Dolgachev, Noam Elkies, Gavril Farkas, Sam Grushevsky, Avi Kulkarni,  Elisa Lorenzo Garc\'{i}a, David Lubicz, Martin Raum, and Christophe Ritzenthaler for helpful conversations during the writing of this article. Furthermore, we thank Brendan Hassett, David Lehavi, Bjorn Poonen, and Jeroen Sijsling for reading a previous version of the manuscript and giving suggestions to improve the article.
\par The first author is supported by MaRDI, funded by the Deutsche Forschungsgemeinschaft (DFG), project number 460135501, NFDI 29/1.
The third author was supported by the Simons Collaboration in Arithmetic Geometry, Number Theory, and Computation via Simons Foundation grant 550033.
\section{Preliminaries} \label{sec:prelims}

\subsection{The Aronhold--Weber formulas in genus $3$} \label{subsec:aronholdweber}
We recall the classical Aronhold--Weber formulas. Let $X/\mathbb{C}$ be a smooth, non-hyperelliptic genus $3$ curve canonically embedded in $\PP^2$, and let $\mu_1, \ldots, \mu_7$ be an Aronhold set of bitangents of $X$. By an appropriate choice of projective coordinates $y_0, y_1, y_2$, one can take
\begin{equation*}
    \mu_1: y_0 = 0,\qquad \mu_2: y_1 = 0,\qquad \mu_3 : y_2 = 0,\qquad \mu_4: y_0 + y_1 + y_2 = 0 \, .
\end{equation*}
For the remaining bitangents, write
\begin{equation} \label{eqn:bitangent_coeffs}
    \mu_{4+i}: a_{i0} y_0 + a_{i1} y_1 + a_{i2} y_2 = 0
\end{equation}
for some $a_{i0}, a_{i1}, a_{i2} \in \CC$, where $i \in \{1,2,3\}$.

One can express the coefficients $a_{ij}$ of the bitangents in \Cref{eqn:bitangent_coeffs} in terms of theta constants. For ease of notation, we identify a theta characteristic
$\big[\hspace{-0.5mm}\begin{smallmatrix}
\delta\\\varepsilon
\end{smallmatrix}\hspace{-0.5mm} \big]$
with the integer $i \in \{0, \ldots, 2^{6}-1\}$ given by
$$
i = 2(\varepsilon_2 + 2\varepsilon_1 + 4\varepsilon_0) + 2^{4}(\delta_2 + 2 \delta_1 + 4\delta_0)\, ,
$$
and write $\vartheta_i \colonequals \thetacha{}{\delta}{\varepsilon}$ for the corresponding theta function with characteristic. Using this notation, one can write \cite[Theorem 2]{Fiorentino} in the following form.

\begin{thm}\label{thm:Fiorentino}
    With notation as above, the coefficients in \Cref{eqn:bitangent_coeffs} are given by
\begin{equation}
 \begin{aligned}
     a_{10} & :=\iu \frac{\vartheta_{33} \vartheta_{5}}{\vartheta_{40} \vartheta_{12}}, &
      a_{11} & :=\iu \frac{\vartheta_{21} \vartheta_{49}}{\vartheta_{28} \vartheta_{56}}, &
      a_{12} & :=\iu \frac{\vartheta_{7} \vartheta_{35}}{\vartheta_{14} \vartheta_{42}}, \\
    a_{20} & :=\iu \frac{\vartheta_{5} \vartheta_{54}}{\vartheta_{27} \vartheta_{40}}, &
      a_{21} & :=\iu \frac{\vartheta_{49} \vartheta_{2}}{\vartheta_{47} \vartheta_{28}}, &
      a_{22} & :=\iu \frac{\vartheta_{35} \vartheta_{16}}{\vartheta_{61} \vartheta_{14}}, \\
    a_{30} & :=-\frac{\vartheta_{54} \vartheta_{33}}{\vartheta_{12} \vartheta_{27}}, &
      a_{31} & :=\frac{\vartheta_{2} \vartheta_{21}}{\vartheta_{56} \vartheta_{47}}, &
      a_{32} &:=\frac{\vartheta_{16} \vartheta_{7}}{\vartheta_{42} \vartheta_{61}}.\\
    \end{aligned}
\end{equation}
\end{thm}
Next, following \cite[Theorem 6.1.7]{Dolgachev} there are linear forms $u_0, u_1, u_2$  determined by the system of linear equations
\begin{equation}\label{eq:aronholdweber}
\begin{aligned}
    &u_0 + u_1 + u_2 + y_0 + y_1 + y_2 = 0,\\
&\frac{u_0}{a_{10}}+ \frac{u_1}{a_{11}}+ \frac{u_2}{a_{12}}
+ k_1a_{10}y_0 + k_1a_{11}y_1 + k_1 a_{12}y_2 = 0,\\
&\frac{u_0}{a_{20}}+ \frac{u_1}{a_{21}}+ \frac{u_2}{a_{22}}
+ k_2a_{20}y_0 + k_2a_{21}y_1 + k_2 a_{22}y_2 = 0,\\
&\frac{u_0}{a_{30}}+ \frac{u_1}{a_{31}}+ \frac{u_2}{a_{32}}
+ k_3a_{30}y_0 + k_3a_{31}y_1 + k_3 a_{32}y_2 = 0.\,\\
\end{aligned}
\end{equation}
where the $y_i$ are variables and the $k_i$ can be found using a certain system of equations described in the same theorem. As shown in \cite[Corollary 2]{Fiorentino}, the normalization factors $k_i$ automatically become $1$ for our choice of the Weber moduli $a_{ij}$.

With the $u_i$ one can write down an equation for $X$ as follows:
\begin{thm}[{\cite[Theorem 6.1.9]{Dolgachev}, \cite[Proposition 2]{Fiorentino}}]
    $X$ is cut out by the equation:
    \begin{equation*}
        4 y_0 u_0 y_1 u_1 = (y_0 u_0 + y_1 u_1 + y_2 u_2)^2 \, .
    \end{equation*}
\end{thm}
Moreover, Aronhold--Weber give the equations of all the $28$ bitangents of $X$. We will print these formulas together with the odd theta characteristics corresponding to the bitangents. To this end, we use  the following odd theta charactistics as an Aronhold system (see \cite{Fiorentino}):
\begin{align*}
\alpha_1 = \frac{1}{2}\begin{bmatrix}
    1 &1 &1 \\
1 & 1& 1
\end{bmatrix},
\alpha_2 = \frac{1}{2}\begin{bmatrix}
    0 & 0 & 1\\
0 & 1 & 1
\end{bmatrix},\alpha_3 = \frac{1}{2}\begin{bmatrix}    
0 & 1& 1\\
0&0&1
\end{bmatrix}, \alpha_4 = \frac{1}{2}\begin{bmatrix}    
1 & 0& 1\\
1 &0&0
\end{bmatrix},\\
\alpha_5 = \frac{1}{2}\begin{bmatrix}    
1 & 0& 0\\
1 & 0& 1
\end{bmatrix}, \alpha_6 = \frac{1}{2}\begin{bmatrix}    
1 &1&0\\
0 & 1&0
\end{bmatrix}
, \alpha_7 =\frac{1}{2}\begin{bmatrix}    
0 & 1& 0\\
1 & 1& 0
\end{bmatrix}.
\end{align*}
Then, by the theory of Aronhold systems \cite{GrossHarris}, the $28=7+\binom{7}{2}$ odd theta characteristics are $\alpha_i$ and $\alpha_{ij} := \sum_{\nu=1}^7\alpha_\nu-\alpha_i-\alpha_j,\, i,j =1, \ldots 7, i\neq j$.

The equations of the 28 bitangents are now given as follows \cite[Theorem 6.1.9]{Dolgachev}:
\begin{small}
\begin{equation}\label{eq:explicitbitangents}
\begin{aligned}   
&\hspace{-20mm} y_0 = 0\quad (\alpha_1),\quad y_1 = 0 \quad (\alpha_2),\quad y_2 = 0 \quad (\alpha_3),\\
&\hspace{-20mm}y_0 + y_1 +y_2 = 0 \quad (\alpha_4),\\
&\hspace{-20mm}a_{i0}y_0+a_{i1}y_1+a_{i2}y_2 = 0 \quad (\alpha_{4+i}), \, i = 1,2,3,\\
&\hspace{-20mm}u_0 = 0\quad(\alpha_{23}),\quad u_1 = 0\quad(\alpha_{13}),\quad  u_2 = 0\quad(\alpha_{12}),\\\\
  &\hspace{-20mm} y_0+y_1+u_2\quad(\alpha_{34}), \quad  y_0+y_2+u_1\quad(\alpha_{24}),\quad y_1 +y_2 + u_0\quad(\alpha_{14}),\\
  &\hspace{-20mm}\frac{u_0}{a_{i0}} + k_i (a_{i1}y_1 + a_{i2}y_2)\quad(\alpha_{1(4+i)}), i = 1,2,3 \\
  &\hspace{-20mm}\frac{u_1}{a_{i1}} + k_i (a_{i0}y_0 + a_{i2}y_2)\quad(\alpha_{2(4+i)}), \, i = 1,2,3 \\
  &\hspace{-20mm}\frac{u_2}{a_{i2}} + k_i (a_{i0}y_0 + a_{i1}y_1)\quad(\alpha_{3(4+i)}), \, i = 1,2,3\\\\
  &\hspace{-20mm}\frac{u_0}{1-k_i a_{i1}a_{i2}} + \frac{u_1}{1-k_i a_{i0}a_{i2}} +\frac{u_2}{1-k_i a_{i0}a_{i1}} \quad (\alpha_{67}, \alpha_{57}, \alpha_{56}), \, i = 1,2,3 \\
\end{aligned}
\end{equation}
\[\frac{u_0}{a_{i0}(1-k_i a_{i1}a_{i2})} + \frac{u_1}{a_{i1}(1-k_i a_{i0}a_{i2})} +\frac{u_2}{a_{i2}(1-k_i a_{i0}a_{i1})}
\quad (\alpha_{45}, \alpha_{46}, \alpha_{47}), \, i = 1,2,3.\]
\end{small}
\begin{rem}
The formulas in \cite[Theorem 6.1.9]{Dolgachev} contain typos that will be corrected in a future edition. We thank Igor Dolgachev for providing us with these corrections.
\end{rem}
\subsection{The Prym canonical map}
This section summarizes results from the articles \cite{Catanese}\cite{BruinSertoz}. Let $C$ be a non-hyperelliptic genus $4$ curve, $\eta\in \Jac(C)[2]\setminus \{0\}$ be a non-trivial 2-torsion point. The Prym canonical map is the map
$$
\phi_\eta: C \longrightarrow \mathbb{P}(\HH^0(\Omega_C(\eta)))\cong \mathbb{P}^2
$$
associated to the linear system $|\Omega_C (\eta)|$. The map $\phi_\eta$ is either a ramified degree $2$ map onto a smooth cubic (the bielliptic case) or a birational map onto a singular plane sextic. (For details, see \cite[\S1]{Catanese}.)
We will discuss the bielliptic case in Section \ref{sec:bielliptic} below. Thus assume now we are in the non-bielliptic case and hence the Prym canonical curve $\im(\phi_\eta)$ is a singular plane sextic embedded in $\mathbb{P}^2$. Let $b_1:S\rightarrow \mathbb{P}^2$ be the repeated blow-up of $\mathbb{P}^2$ in the singular points of the Prym canonical curve. Denote by $E_1, \ldots, E_k$ the total transform of the exceptional divisors of the successive blow-ups. On $S$ we can consider the line bundle $\mathcal{L} = b_1^*\mathcal{O}_{\PP^2}(3)\otimes \mathcal{O}(-\sum_{j=1}^k (r_j-1)E_j)$ where $r_j$ is the multiplicity of the singular point blown up in the $j$th step. The line bundle $\mathcal{L}$ induces a map $b_2: S\rightarrow \mathbb{P}(\HH^0(\Omega_C))\cong \mathbb{P}^3$ whose image, which we denote by $\Gamma_\eta$, is a cubic symmetroid \cite[p.~37]{Catanese}. We recall the definition of a cubic symmetroid.
\begin{definition}\label{def:symmetroid}
A \textit{cubic symmetroid} is a cubic surface $V$ which is the vanishing scheme of the determinant of a symmetric $3\times 3$ matrix of linear forms. That is, there exists a symmetric $3 \times 3$ matrix $A$ with entries $a_{ij}\in k[x_0, \ldots, x_3]$ homogeneous of degree $1$ such that $V = \mathcal{V}(\det(A))$.
\end{definition}
We define a rational map $c: \mathbb{P}^2\dashrightarrow \Gamma_\eta$ as the composition of $b_1^{-1}$ with the map $b_2:S \rightarrow \Gamma_{\eta}$.
The situation is summarized by the following commutative diagram.
\begin{equation}
\label{eqn:cayley_cubic}
\begin{tikzcd}
& & S\arrow{dr}{b_2}\arrow{dl}[swap]{b_1}\\
C \ar[hook]{r}{\phi_\eta}& \mathbb{P}^2 \arrow[rr, dashrightarrow, "c"] &&\Gamma_\eta \subseteq \mathbb{P}^3
\end{tikzcd}
\end{equation}
The map $b_2: S\rightarrow \Gamma_\eta$ is a blow-up; this follows from the observations in \cite[p. 37]{Catanese}. Furthermore, the composition $c \circ \phi_\eta: C \to \PP^3$ is equal to the canonical embedding and hence independent of $\eta$. This will be shown in the proof of Lemma \ref{lem: psirightinv} below.
\par We will describe the situation for a generic $C$. In this case, the Prym canonical image $\im(\phi_\eta)$ has six nodes that are the pairwise intersection of four lines, say $n_1, \ldots, n_4$. The surface $S$ is the blow-up in these six nodes. The strict transforms of the $n_i$ under $b_1$ are $(-2)$-curves. These are contracted by $b_2$ to $A_1$ singularities on $\Gamma_\eta$.
\par There are two possible degenerate cases. First, it could happen that two of the lines $n_i$ coincide. In this case, the two corresponding $A_1$-singularities combine to form a single $A_3$-singularity. Then $\Gamma_\eta$ has three singular points: Two of type $A_1$ and one of type $A_3$.
\par Second, three of the lines $n_i$ could coincide. In that case, $\Gamma_\eta$ has one singularity of type $A_1$ and one of type $A_5$. For more details, see \cite[p. 33]{Catanese}.
\par The main consequence of Catanese's construction is the following theorem that he attributes to Wirtinger--Coble--Recillas.
\begin{thm} The map $\eta \mapsto \Gamma_\eta$ gives a bijection between:
\begin{itemize}
    \item $\Jac(C)[2]\setminus \{0\}$.
    \item Irreducible cubic symmetroids containing the canonical image of $C$ in $\mathbb{P}^3$.
\end{itemize}
\end{thm}
\begin{proof}
    See \cite[Theorem 1.5]{Catanese}.
\end{proof}
\subsection{Two linear maps} \label{subsec:2linear}
 We will now define a linear map $\varphi$ that maps certain quadratic forms on the right $\mathbb{P}^3$ to quadratic forms on the left $\mathbb{P}^2$ in Diagram (\ref{eqn:cayley_cubic}). We begin by considering the following cartesian diagram. 
\begin{equation} \label{eqn:phi_diagram}
\begin{tikzcd}
    S^2 \HH^0(\Omega_C(\eta))\times_{\HH^0(\Omega_C^{\otimes 2})} S^2 \HH^0(\Omega_C)\arrow{r}\arrow[hookrightarrow]{d}\arrow{r}{} & S^2 \HH^0(\Omega_C(\eta))\arrow[hookrightarrow]{d}\\
     S^2 \HH^0(\Omega_C)\arrow{r} & \HH^0(\Omega_C^{\otimes 2})    
\end{tikzcd}
\end{equation}
where the maps $S^2\HH^0(\Omega_C (\eta)) \rightarrow \HH^0(\Omega_C^{\otimes 2}),\, S^2\HH^0(\Omega_C) \rightarrow \HH^0(\Omega_C^{\otimes 2})$ are given by multiplication. The injectivity of the right vertical arrow expresses the fact that there are no quadrics vanishing on $\im(\phi_\eta)$.
\begin{definition}
Following Lehavi \cite[p. 2]{Lehavi} we define $V_{C, \eta}\subset S^2 \HH^0(\Omega_C)$ as the image of the embedding
$$S^2 \HH^0(\Omega_C (\eta))\times_{\HH^0(\Omega_C^{\otimes 2})} S^2 \HH^0(\Omega_C) \hookrightarrow S^2 \HH^0( \Omega_C).$$
Identifying $S^2 \HH^0(\Omega_C (\eta))\times_{\HH^0(\Omega_C^{\otimes 2})} S^2 \HH^0(\Omega_C)$ with its image $V_{C,\eta}$, we define Lehavi's map $\varphi$ to be the projection map onto the first factor
$$\varphi: V_{C, \eta} \rightarrow S^2 \HH^0(\Omega_C (\eta))\,.$$
\end{definition}
From the facts that $\dim (S^2 \HH^0(\Omega_C))=10,\, \dim (S^2 \HH^0(\Omega_C(\eta)))=6$ and
$\dim (\HH^0(\Omega_C^{\otimes 2}))=9$, and that the bottom map $S^2 H^0(\Omega_C) \to H^0(\Omega_C^{\otimes 2})$ is surjective, then $\dim(V_{C, \eta})=10+6-9=7$. Furthermore, it will later turn out to be useful that $\ker(\varphi)$ is one-dimensional and generated by the quadratic form $Q$ vanishing on $C$.
\par Next, we define a right inverse $\psi$ for $\varphi$ using the following lemma.
\begin{lem}\label{lem: O2iso}
Assume that $\Gamma_\eta$ is generic. Let us denote by $N_i \subset S,\, i=1, \ldots, 4$ the exceptional divisors of $b_2$. Then
$b_1^*\mathcal{O}(2)\cong b_2^* \mathcal{O}(2)\otimes \mathcal{O}_S(-\sum N_i)$.
If $\Gamma_\eta$ is not generic, then the $N_i$ have to be taken with appropriate multiplicities.
\end{lem}
\begin{proof}
Consider the curves $n_i=b_1(N_i)$. Then, by the discussion following Definition \ref{def:symmetroid}, the $n_i \in \mathbb{P}^2$ are lines such that the 6 singular points of $\im(\phi_\eta)$ are the intersection points of pairs from $\{n_i \mid i = 1,\ldots, 4\}$.
\par By definition of $b_2$ we have an isomorphism
$$b_2^* \mathcal{O}(1) \cong \mathcal{L} = b_1^*\mathcal{O}_{\PP^2}(3)\otimes \mathcal{O}_S\left(-\sum_{j=1}^6 E_j\right)\,.$$
These two facts imply the lemma.
\end{proof}
We define 
$$
W_{\eta}\colonequals\HH^0\left(S, b_2^*\mathcal{O}(2)\otimes \mathcal{O}_S\!\left(-\sum N_i\right)\right)
$$
and consider it as a vector subspace of $
\HH^0\left(S, b_2^*\mathcal{O}(2)\right) \cong 
S^2 \HH^0(C, \Omega_C)$.
Then, $W_\eta$ is the vector space of quadrics vanishing on the singular locus of $\Gamma_\eta$. One has $\dim(W_\eta) = 6$ because there is an isomorphism $
\psi: S^2 \HH^0(C, \Omega_C (\eta))\stackrel{\sim}{\longrightarrow} W_\eta
$
which is defined to be the composition of isomorphisms
$$
\resizebox{\textwidth}{!}{$
S^2 \HH^0(C, \Omega_C (\eta)) \cong \HH^0(\mathbb{P}^2, \mathcal{O}(2))\overset{b_1^*}{\stackrel{\sim}{\rightarrow}} \HH^0(S, b_1^*\mathcal{O}(2))\cong  \HH^0\left(S, b_2^*\mathcal{O}(2)\otimes \mathcal{O}_S\left(-\sum_i N_i\right)\right) \, .
$}
$$
Geometrically this means the following: for any $q\in S^2\HH^0(C, \Omega_C(\eta))$, the image $\psi(q)$ is the unique (up to scaling) quadratic form that cuts out the curve
$$
c(\mathcal{V}(q)) = b_2(b_1^{-1}(\mathcal{V}(q)))\subset \Gamma_\eta \, .
$$
Notice that the map $\psi$ is only well-defined up to multiplication by a scalar because of the arbitrary choice of an isomorphism $b_1^*\mathcal{O}(2)\cong b_2^* \mathcal{O}(2)\otimes \mathcal{O}_S(-\sum N_i)$. We resolve this ambiguity in the proof of the following lemma which shows that $\psi$ is a right-inverse of $\varphi$ (up to scalar). Indeed, we rescale $\psi$ so that $\varphi \circ \psi = \id$.
\begin{lem}\label{lem: psirightinv} ${}$
\begin{enumerate}[(i)]
\item $W_\eta$ is a subspace of $V_{C, \eta}$.  
\item $\varphi \circ \psi= \alpha\cdot \id\,$ for some $\alpha\in k^\times$.
\end{enumerate}
\end{lem}
\begin{proof}
We start by considering the diagram
\begin{equation}
    \label{eqn:proofdiag}
 \begin{tikzcd}
     & C\arrow[swap]{ddl}{\phi_\eta} \arrow{ddr}\arrow[hook]{d} \\
     & S\arrow{dl}{b_1} \arrow[swap]{dr}{b_2} \\
 \mathbb{P}^2 \arrow[rr, dashrightarrow, "c"] &&\mathbb{P}^3
 \end{tikzcd}
 \end{equation}
 which commutes by \cite[Equation 1.13]{Catanese} and (\ref{eqn:cayley_cubic}). Now recall from Lemma \ref{lem: O2iso} that the definition of $\psi$ is based on the isomorphism
 $$b_1^*\mathcal{O}(2)\cong b_2^* \mathcal{O}(2)\otimes \mathcal{O}_S\!\left(-\sum_i N_i\right)\,.$$
 Restricting this isomorphism to $C$ gives an isomorphism $\Omega_C(\eta)^{\otimes 2} \cong \Omega_C^{\otimes 2}\otimes \mathcal{O}_C\,.$
 (Notice that $C$ cannot go through the nodes of $\Gamma_\eta$ because then $Q\cap \Gamma_\eta$ would be singular.)
 This, together with the commutativity of (\ref{eqn:proofdiag}), implies that the diagram
 $$
 \begin{tikzcd}
     &  S^2 \HH^0(\Omega_C(\eta)) \arrow{d}\arrow[swap]{dl}{\psi}\\
     S^2\HH^0(\Omega_C)\arrow{r} & \HH^0(\Omega_C^{\otimes 2})
 \end{tikzcd}
 $$
 commutes after replacing $\psi$ by a suitable scalar multiple, and we assume that to be done.  From the definition of a fiber product we obtain that $W_\eta = \im(\psi)\subseteq V_{C, \eta}$, proving (i).
 \par  Since the diagram \ref{eqn:phi_diagram} commutes as well, and the map $S^2 \HH^0(\Omega_C (\eta)) \rightarrow \HH^0(\Omega_C^{\otimes 2})$ is injective, we conclude that $\varphi \circ \psi = \id \, ,$ which proves the lemma.
\end{proof}




 
\subsection{Milne's bijection}

Let $\widetilde{C} \to C$ be the \'{e}tale double cover associated to $\eta \in \Jac(C)[2]\setminus \{0\}$, so $\widetilde{C}$ has genus 7. Since the Prym variety $\Prym(\widetilde{C}/C)$ is principally polarized, then, up to quadratic twist, $\Prym(\widetilde{C}/C)$ is isomorphic to the Jacobian of some genus $3$ curve $X$ (see \cite{OortUeno} and \cite{Beauville}). (The Schottky--Jung relations imply that the Prym is indecomposable.)  Since we have an isomorphism
\begin{align*}
\HH^0(X, \Omega_X) &\cong \HH^0\!\left(\Jac(X), \Omega_{\Jac(X)}\right)\\
&\cong \HH^0\!\left(\Prym(\widetilde{C}/C), \Omega_{\Prym(\widetilde{C}/C)}\right) \cong \HH^0(C, \Omega_C (\eta))
\end{align*}
by the definition of the Prym variety, the curve $X$ canonically maps to $\mathbb{P}(\HH^0(C, \Omega_C(\eta)))\cong \mathbb{P}^2$. 

\begin{rem} The Recillas trigonal construction (see \cite[\S50]{coble} or \cite[Lemma 5.6]{BruinSertoz}) gives an explicit geometric construction of $X$ as a plane quartic (generic case), but this is not used in the present article.
\end{rem}
\par In order to state the main theorem of this section we need the following definition. Given a tritangent plane $H$ of $C$, then $H.C = 2D$ where $D$ is a divisor of degree $3$. We say that a pair of tritangents $H, H'$ \textit{differ by} $\eta \in \Jac(C)[2]$ if $D - D'$ is linearly equivalent to $\eta$, where $H.C = 2D$ and $H'.C = 2D'$.

\begin{thm} \label{thm:Milne}
If $X$ is a plane quartic, there is a bijection
$$
\left\{
\begin{tabular}{c}
Pairs of tritangent planes of $C$ differing by $\eta$
\end{tabular}
\right\}
\overset{\sim}{\rightarrow} \left\{
\begin{tabular}{c}
Bitangents of $X$
\end{tabular}
\right\}
$$
such that for any pair $H, H'$ viewed as elements of $\HH^0(C, \Omega_C)$
mapping to a bitangent $\ell \subset \mathbb{P}(\HH^0(C, \Omega_C(\eta)))$ viewed as an element of $\HH^0(C, \Omega_C(\eta))$, we have:
\begin{enumerate}[(i)]
    \item The product $HH'\in S^2 \HH^0(C, \Omega_C)$ lies in $V_{C, \eta}$.
    \item There exists $\lambda \in k^\times$ such that
\begin{equation}
    \label{eqn:Milne}
    \varphi(H H')=\lambda \ell^2.
\end{equation}
\end{enumerate}
\end{thm}
\begin{proof}
    The bijection between pairs of tritangent planes differing by $\eta$ and bitangents of $X$ was discovered by W. P. Milne \cite{Milne} (see \cite[Theorem 7.1]{BruinSertoz} for a modern proof). The assertion (i) is due to Lehavi \cite[p. 2]{Lehavi}. To prove (ii) one can use the argument from \cite[Theorem 7.1]{BruinSertoz}, the geometric interpretation of the map $\psi$, and Lemma \ref{lem: psirightinv}.
\end{proof}
\begin{rem}\label{rem: MilneHyp} A similar statement holds when $X$ is a hyperelliptic curve. The main difference is that one has to replace the bitangents of $X$ with the lines through the pairs of Weierstrass points of $X$ in the bijection (see \cite[Section 7.2]{BruinSertoz}).

\end{rem}

Milne's theorem will play a key role in reconstructing the genus $4$ curve because it links the map $\varphi$ to the Prym. Since the bitangent lines of $X$ can be computed with classical formulas, namely the Schottky--Jung relations \cite[Theorem 1]{FarkasRauch}, and the Aronhold-Weber formulas (see \cref{subsec:aronholdweber}), the information on the righthand side of the identity
$$\varphi(H H')=\lambda \ell^2$$
can be readily computed, except for the unknown constant $\lambda$. A central step of our method is to compute the map $\varphi$ by interpolating this identity through ten pairs of tritangent planes and their corresponding bitangents.

\subsection{The bielliptic case}\label{sec:bielliptic}
We will explain now the modifications that have to be made in the bielliptic case. All the facts in this section are due to Catanese and Bruin--Sert\"oz; see \cite[(1.9)]{Catanese}, \cite[\S4.2]{BruinSertoz} for proofs. Let $C$ be a smooth non-hyperelliptic genus $4$ curve and $\eta\in \Jac(C)[2]\setminus \{0\}$. We will say that the pair $(C, \eta)$ is \emph{bielliptic} if there is a degree two map $\pi: C\rightarrow E$ onto a smooth genus $1$ curve and a two-torsion point $\eta_0 \in \Jac(E)$ such that $\pi^* \eta_0=\eta$.
\par Then $(C, \eta)$ is bielliptic if and only if the Prym canonical map $\phi_\eta: C\rightarrow \mathbb{P}^2$ factors through a degree two map onto a smooth plane cubic $E\subset \mathbb{P}^2$. In this situation, one can still obtain $\Gamma_\eta$ as the cone over $E$ with vertex corresponding to the one-dimensional subspace
$$\pi^* \HH^0(E, \Omega_E)\subset \HH^0(C, \Omega_C)\,.$$
One still has $C\subset \Gamma_\eta$ and the map $\pi: C\rightarrow E$ is induced by projecting away from the vertex of $\Gamma_\eta$.
\par On the other hand, $\Gamma_\eta$ is a cubic symmetroid since the two-torsion point $\eta_0$ defines a symmetric determinantal equation for $E\subset \mathbb{P}^2$ \cite[Section 4.1.3]{Dolgachev}. But this cubic symmetroid is degenerate in the following sense: when writing
$$\Gamma_\eta=\mathcal{V}\left(\det\left(\sum_{i=0}^3 A_i x_i\right)\right)$$
with $A_i\in \Mat_{3,3}(k)$ symmetric, the matrices $A_i$ are linearly dependent. Also, the map $c$ does not exist in the bielliptic case because $\Gamma_\eta$ is not a rational surface.
\par Nevertheless, the linear maps $\varphi, \psi$ are still defined. Indeed, $\varphi$ was defined unconditionally. On the other hand, the map $$\psi: S^2 \HH^0(C, \Omega_C(\eta)) \rightarrow W_\eta \subset S^2 \HH^0(C, \Omega_C)$$
can be defined by taking $W_\eta$ to be the set of quadrics that are singular at the vertex of $\Gamma_\eta$. The map $\psi$ is then the map sending a conic to the affine cone over it.
\par All the results from the previous sections still hold in the bielliptic case. However, the proof of Lemma \ref{lem: psirightinv} needs a separate argument which we give now. Indeed, instead of diagram \ref{eqn:proofdiag} we consider the diagram
$$\begin{tikzcd}
&C \arrow{dl}[swap]{\phi_\eta}\arrow{dr}\\
\mathbb{P}^2 && \mathbb{P}^3 \arrow[ll, dashrightarrow]
\end{tikzcd}$$
where the map $\mathbb{P}^3\dashrightarrow \mathbb{P}^2$ is the projection away from the vertex of $\Gamma_\eta$. The diagram commutes because, as noted above, this projection induces the map $C\rightarrow E$. The rest of the proof follows the same line of reasoning as in Lemma \ref{lem: psirightinv}.
\subsection{Tritangent planes and theta derivatives}
In order to compute of the tritangent planes on the left-hand side of Equation (\ref{eqn:Milne}), we require formulas that express them purely in terms of the theta constants. To derive such formulas one begins with the following well-known connection between tritangent planes and theta derivatives.
\begin{thm}
    Let $C$ be a smooth curve of genus $4$ over $\mathbb{C}$ with small period matrix $\tau$.
    \par For any odd theta characteristic $\begin{bmatrix} \delta\\\varepsilon     \end{bmatrix} \in \frac{1}{2}\mathbb{Z}^{8}/\mathbb{Z}^8$, the equation
    \begin{equation}
        \sum_{i=1}^4 \frac{\partial \thetacha{}{\delta}{\varepsilon}}{\partial z_i}(0, \tau) x_i=0
    \end{equation}
    defines a tritangent plane for the canonical image of $C$ in the $\mathbb{P}^3$ with coordinates $x_1,\ldots, x_4$. Here the basis for $\HH^0(C, \Omega_C)$ must be the basis induced by the isomorphism
    $$\Jac(C)\cong \mathbb{C}^g/(\mathbb{Z}^g + \tau \mathbb{Z}^g)\,.$$
 \end{thm}
\begin{proof}
See, e.g., \cite[p.~5]{Fiorentino}.
\end{proof}
Next, we want to replace the theta derivatives by theta constants. For this we begin with the same approach as in the genus $3$ case \cite{Fiorentino}: one chooses $5$ tritangent planes and applies a coordinate transform that puts them into the standard form
$$x_i=0, \qquad x_1+x_2+x_3+x_4=0.$$
Then Cramer's rule expresses the equations for the other tritangent planes in terms of determinants of Jacobi matrices of theta functions (see \cite[Equation (17)]{Fiorentino} or Equation (\ref{eqn:tritangents}) in the present article).
The next section explains how these determinants can be expressed in terms of theta constants.
\subsection{The generalized Jacobi derivative identity}
In this section, we introduce Fay's generalization of Jacobi's derivative formula. It expresses the Jacobian determinant of an azygetic system of odd theta functions evaluated at $0$ as a polynomial in the theta nullvalues.
\par First, we will recall some definitions from the theory of theta characteristics. A system of theta characteristics $c_1, c_2, c_3\in \frac{1}{2}\mathbb{Z}^{2g}$ is called an \emph{azygetic} triple if
$$e_*(c_1+c_2+c_3)=-e_*(c_1)e_*(c_2)e_*(c_3) \, ,$$
where $e_*: \frac{1}{2}\mathbb{Z}^{2g} \rightarrow \{\pm 1\}$ is the parity map $\begin{bmatrix}
    \delta\\ \varepsilon
\end{bmatrix} \mapsto (-1)^{4 \delta^t \varepsilon}$.
More generally, an arbitrary system of characteristics $c_1, \ldots, c_n$ is called \emph{azygetic} if any triple contained in $c_1, \ldots, c_n$ is an azygetic triple.
\par A system of theta characteristics is called \emph{essentially independent} if every sum of a subset of even cardinality of the system is non-zero.
\begin{definition}
A \emph{special fundamental system} is a system of $2g+2$ characteristics $m_1, \ldots, m_g, n_1, \ldots, n_{g+2}\in \frac{1}{2}\mathbb{Z}^g$ such that:
\begin{enumerate}

\item[i)] $m_1, \ldots, m_g, n_1, \ldots, n_{g+2}$ is azygetic.
\item[ii)] The characteristics $m_1, \ldots, m_g$ are odd.
\item[iii)] The characteristics $n_1, \ldots, n_{g+2}$ are even.
\end{enumerate}
\end{definition}

We will now define the left-hand side of the generalized Jacobi derivative identity.
\begin{definition}
Let $m_1, \ldots, m_g\in \frac{1}{2}\mathbb{Z}^{2g}$ be a system of odd characteristics. The \emph{Jacobian nullvalue} of $m_1, \ldots, m_g$ is defined to be the function on the Siegel upper half space
$$D(m_1, \ldots, m_g): \mathfrak{H}_g \longrightarrow \mathbb{C}$$
given by the formula
$$D(m_1, \ldots, m_g)(\tau)=\pi^{-g} \det\left(\left(\frac{\partial\vartheta[m_i]}{\partial z_j}\right)_{i,j=1,\ldots ,g }\right)(0, \tau) \, . $$
\end{definition}
Now if $m_1, \ldots, m_g$ is azygetic and essentially independent then, for $g\leqslant 5$, $D(m_1, \ldots, m_g)$ can be expressed as the following polynomial in the theta nullvalues.
\begin{thm}\label{ThmJacobi}
Let $g\leqslant 5$ and $M=\{m_1, \ldots, m_g\}$ be an azygetic essentially independent system of odd theta characteristics. Then for all $\tau\in \mathfrak{H}_g$,
\begin{equation} \label{eqn:JacobiDeriv}
    D(m_1, \ldots, m_g)(\tau)=\sum_N \pm \prod_{i=1}^{g+2} \vartheta[n_i](0, \tau)
\end{equation}
where the sum runs over all sets $N=\{n_1, \ldots, n_{g+2} \}$ such that $M\cup N$ is a special fundamental system.
\par Furthermore, the signs $\pm$ are explicit, unique and independent of $\tau$.
\end{thm}
\begin{proof}
See \cite{Frobenius} for $g=4$ and \cite{Fay} for $g\leqslant 5$.
\end{proof}
\begin{rem}
The number of terms on the right of Fay's generalized Jacobi derivative identity is given by the following table.\\
\begin{center}
\begin{tabular}{c|c|c|c|c|c}
$g$ & 1 & 2 &3 &4 &5\\
\hline
no. of terms & 1 & 1& 1& 2 & 8
\end{tabular}
\end{center}
\end{rem}
\begin{rem}\label{rem:ConjIgusa} Igusa \cite{Igusa1980Jacobi}\cite{IgusaMultiplicity} has a conjectural generalization of Fay's generalized Jacobi identity to arbitrary $g\geqslant 6$. However, as it is known that in this case a Jacobian nullvalue $D(m_1, \ldots, m_g)$ cannot be a polynomial in the theta constants \cite[p. 12]{Fay}, Igusa's conjectural formula has a sum of Jacobian nullvalues on the left-hand side.
\par It would be an interesting problem to use Igusa's conjectural identity to express Jacobian nullvalues as rational functions in the theta constants. Such expressions must exist by general principles \cite[Theorem V.9]{IgusaThetaBook}, but they seem to be unknown. For work beyond the azygetic case, see \cite{Salvati-Manni}.
\end{rem}

\section{Reconstructing the curve} \label{sec:main}
\subsection{An auxiliary set of odd theta characteristics}\label{sub:oddtheta}
In order to use the generalized Jacobi identity (Theorem \ref{ThmJacobi}) for the computation of the pairs of tritangent planes in Equation (\ref{eqn:Milne}) we will need a set of odd theta characteristics satisfying certain properties. It is not possible to work with an Aronhold system in analogy with genus $3$ because Aronhold systems in genus $4$ consist of even theta characteristics \cite{GrossHarris}.

Instead, we observe that on the left-hand side of the identity we have the Jacobian nullvalue of a set of four azygetic essentially independent odd theta characteristics. Trying to optimize the use of the generalized Jacobi identity leads us to choose the following auxiliary set of odd theta characteristics.
\begin{lem} \label{lem:special_characteristics}
Let $\eta\in \frac{1}{2}\mathbb{Z}^8/\mathbb{Z}^8$ be arbitrary. There exist odd theta characteristics $\xi_1, \ldots, \xi_5,\, \chi_1, \chi_1', \ldots, \chi_{10}, \chi_{10}' \in \frac{1}{2}\mathbb{Z}^8/\mathbb{Z}^8$ such that
\begin{itemize}
    \item[i)] $\xi_1, \ldots, \xi_5$ is azygetic and essentially independent.
    \item[ii)] For all $i\in \{1,\ldots, 10\}$ the systems $\xi_1, \ldots, \xi_4, \chi_i$ and $\xi_1, \ldots, \xi_4, \chi_i'$ are azygetic and essentially independent.
    \item[iii)] For all $i$ one has $\chi_i-\chi_i'=\eta$.
\end{itemize}
\end{lem}
\begin{proof}
Without loss of generality we can assume that $\eta=\frac{1}{2}\begin{bmatrix}
    0 & 0 & 0 & 0\\
    1 & 0 & 0 & 0
\end{bmatrix}$. (This is the convention used in, e.g., \cite{FarkasRauch}.) An explicit answer is then given by:

\begin{align*}
\xi_1 &=\frac{1}{2}\begin{bmatrix} 1 & 1 & 1 & 0 \\ 1 & 1 &  1 & 0 \end{bmatrix}, \quad
\xi_2 =\frac{1}{2}\begin{bmatrix} 1 & 0 & 1 & 0 \\ 0 & 0 &  1 & 0 \end{bmatrix}, \quad
\xi_3 =\frac{1}{2}\begin{bmatrix} 1 & 1 & 1 & 0 \\ 0 & 0 &  1 & 0 \end{bmatrix},\\
\xi_4 &=\frac{1}{2}\begin{bmatrix} 1 & 0 & 1 & 0 \\ 0 & 1 &  1 & 0 \end{bmatrix}, \quad
\xi_5 =\frac{1}{2}\begin{bmatrix} 0 & 1 & 1 & 0 \\ 0 & 1 &  0 & 0 \end{bmatrix},
\end{align*}
\begin{align*}
\chi_1 &=\frac{1}{2}\begin{bmatrix} 0 & 1 & 1 & 1 \\ 0 & 1 &  1 & 1 \end{bmatrix}, \quad
\chi_2 =\frac{1}{2}\begin{bmatrix} 0 & 1 & 0 & 1 \\ 0 & 1 &  0 & 0 \end{bmatrix},\quad
\chi_3 =\frac{1}{2}\begin{bmatrix} 0 & 1 & 0 & 0 \\ 0 & 1 &  0 & 1 \end{bmatrix},\\
\chi_4 &=\frac{1}{2}\begin{bmatrix} 0 & 1 & 0 & 0 \\ 0 & 1 &  1 & 0 \end{bmatrix},\quad
\chi_5 =\frac{1}{2}\begin{bmatrix} 0 & 1 & 1 & 0 \\ 0 & 1 &  0 & 0 \end{bmatrix}, \quad 
\chi_6 = \frac{1}{2}\begin{bmatrix} 0 & 1 & 1 & 0 \\ 0 & 1 &  0 & 1 \end{bmatrix},\\
\chi_7 &=\frac{1}{2}\begin{bmatrix} 0 & 1 & 0 & 0 \\ 0 & 1 &  1 & 1 \end{bmatrix}
, \quad
\chi_8 =\frac{1}{2}\begin{bmatrix} 0 & 1 & 0 & 1 \\ 0 & 1 &  1 & 0 \end{bmatrix}, \quad
\chi_9 =\frac{1}{2}\begin{bmatrix} 0 & 1 & 1 & 1 \\ 0 & 1 &  0 & 0 \end{bmatrix}, \\
\chi_{10} &=\frac{1}{2}\begin{bmatrix} 0 & 1 & 0 & 0 \\ 0 & 1 &  0 & 0 \end{bmatrix}
\end{align*}
and $\chi_i' = \chi_i +\eta$ for $i \in \{1,\ldots,10\}$.
\end{proof}
\begin{rem} The answer in the previous proof is found by choosing $\xi_1, \ldots, \xi_5$ suitably and then solving a system of equations over $\mathbb{F}_2$ for $\chi_1, \chi_1', \ldots, \chi_{10}, \chi_{10}'$ allowing $\eta$ to be arbitrary. In the end one transforms the characteristics such that $\eta$ becomes $\frac{1}{2}\begin{bmatrix}
    0 & 0 & 0 & 0\\
    1 & 0 & 0 & 0
\end{bmatrix}$.
\par One can still see the (affine) linear structure in the shape of the characteristics $\chi_i, \chi_i'$. Indeed, the $\chi_i$ are all the odd theta characteristics whose left $2\times 2$ block is
$$\begin{bmatrix}
0 & 1\\
0 & 1
\end{bmatrix} \, .$$
From this observation one sees that one cannot add further pairs $\chi_i, \chi_i'$ with the required properties to the system.
\end{rem}
For the rest of the article, we define $H_i, H_i' \in \HH^0(C, \Omega_C)$ to be the linear forms cutting out the tritangent planes corresponding to $\chi_i, \chi_i'$.
By construction, for any $i=1, \ldots, 10$ the tritangent planes $H_i$ and $H_i'$ differ by $\eta$. We denote by $\ell_i \in \HH^0(C, \Omega_C(\eta))$ the linear form cutting out the bitangent of $X$ that corresponds to $(H_i, H_i')$ under Milne’s bijection (Theorem \ref{thm:Milne}).
\par The main point of the special conditions in Lemma 5.1 is that they guarantee that the tritangent planes $H_i, H_i'$ can be written with simple formulas in terms of the theta constants, as we now explain. Indeed, we choose the coordinate system for $\mathbb{P}^3$ such that the tritangent planes corresponding to the odd characteristics $\xi_i,\, i=1, \ldots, 5$ are in normal form
$$
x_i=0, \quad i\in \{0,\ldots 3\},$$
$$  x_0+x_1+x_2+x_3=0\,.
$$
Then, with the same argument as in \cite[Equation (17)]{Fiorentino} one sees that for any $i\in \{1, \ldots, 10\}$ the tritangent plane $H_i$ is given by the equation

\begin{equation}
\label{eqn:tritangents}
\begin{split}
\frac{D(\chi_i, \xi_2, \xi_3, \xi_4)  }{D(\xi_5, \xi_2, \xi_3, \xi_4)} x_0+&\frac{D( \xi_1, \chi_i,\xi_3, \xi_4)  }{D( \xi_1, \xi_5, \xi_3, \xi_4)} x_1+\\ &\frac{D( \xi_1,\xi_2, \chi_i, \xi_4)  }{D( \xi_1,\xi_2, \xi_5, \xi_4)} x_2+
\frac{D( \xi_1,\xi_2,\xi_3, \chi_i)  }{D( \xi_1,\xi_2,\xi_3, \xi_5)} x_3=0
\end{split}
\end{equation}
in this coordinate system. An analogous formula with $\chi_i$ replaced by $\chi_i'$ gives the tritangent planes $H_i'$.
\par By Lemma \ref{lem:special_characteristics} the set $\xi_1, \ldots, \xi_5$ and all the sets $\xi_1, \ldots, \xi_4, \chi_i$ as well as $\xi_1, \ldots, \xi_4, \chi_i'$ for $i=1, \ldots, 10$ are azygetic and essentially independent. This implies that the Jacobian nullvalue in our equations for $H_i, H_i'$ can be expressed in terms of theta constants via Theorem \ref{ThmJacobi}.
\begin{rem}\label{rem:Dzer}
The denominators in Equation (\ref{eqn:tritangents}) are non-zero for a generic curve. But it can happen that they vanish: such curves must exist because the corresponding locus in the Satake compactification of the Siegel moduli space is given by the vanishing locus of a non-zero polynomial in the theta constants on the (open) Torelli locus. By \cite[Theorem 2.3]{Faltings} such a polynomial is a section of an ample line bundle. Furthermore, since the boundary of the Torelli locus in the Satake compactification has codimension two, this locus must be non-empty. Nevertheless, the authors do not know an explicit example of such a curve.

\end{rem}

\subsection{Finding the quadric}\label{subsec:quad}


In this section, we describe how to recover the quadric $Q$ containing the canonical embedding of $C$. The first step is obtaining the equation for the map $\varphi$ defined in the diagram (\ref{eqn:phi_diagram}). To do this, we construct and solve a linear system as follows.

Recall from Theorem \ref{thm:Milne} that
$$
\varphi(H H')=\lambda \ell^2
$$
for any pair of tritangent planes $(H, H')$ that maps to a bitangent $\ell$ under Milne’s bijection.

\begin{lem}\label{lem:gen_set}
For a generic genus $4$ curve, the elements $H_i H_i'\in V_{C, \eta}$, for $i=1, \ldots, 10$ form a generating set of $V_{C, \eta}$.
\end{lem}
\begin{proof}
    Since the property is open in the moduli space and the moduli space is irreducible, it suffices to show that one genus $4$ curve satisfies the condition of the lemma. We will give an example in Example \ref{ex:finite_field}.
\end{proof}
Using the knowledge of the equations of the pairs of tritangent planes $H_i, H_i'$ and the corresponding bitangents $\ell_i$ for $i=1, \ldots, 10$, we know everything in the equations
\begin{equation}\label{eqn:Milne10}
 \varphi(H_i H_i') = \lambda_i \ell_i^2,\ i=1, \ldots, 10
\end{equation}
except for the scalars $\lambda_i$. The linear dependencies satisfied by the $H_i H_i'$ yield a homogeneous linear system of equations with unknowns $\lambda_i$. (Recall that $H_i H_i' \in V_{C,\eta}$ and $\dim(V_{C,\eta})=7$, as mentioned in Section \ref{subsec:2linear}.)
Solving this linear system, we recover the linear map $\varphi$, and thus the quadric $Q$ from $\ker(\varphi)$.
\begin{lem}\label{lem:proof_example}
For a generic genus $4$ curve the vector $(\lambda_1, \ldots, \lambda_{10})$ is uniquely determined from this linear system up to a scalar.
\end{lem}
\begin{proof}
With the same argument as in Lemma \ref{lem:gen_set} it suffices to check that it holds for Example \ref{ex:finite_field}.
\end{proof}

\subsection{Finding the cubic symmetroid $\Gamma_\eta$}\label{sec:cubic}
In this section, we explain how the knowledge of the map
$$\varphi: V_{C, \eta} \longrightarrow S^2\HH^0(C, \Omega_C(\eta))$$
can be used to recover the equations of the curve $C$. As we know that $\varphi$ recovers $Q$, it remains to explain the calculation of a cubic containing $C$.
We begin by giving a connection between the natural right-inverse $\psi$ of $\varphi$ and the Cayley cubic $\Gamma_\eta$. For this purpose, we will introduce a sextic form $G$ which vanishes with multiplicity two on $\Gamma_\eta$. In this definition we use the abbreviations $V=\HH^0(C, \Omega_C),\, W=\HH^0(C, \Omega_C (\eta))$.
\par 
    Let $G$ be the homogeneous form that is given by the following composition of polynomial maps
\begin{equation} \label{eqn: defG}
    V^\vee \rightarrow S^2(V^\vee) \cong (S^2 V)^\vee \rightarrow V_{C, \eta}^\vee \stackrel{\psi^\vee}{\longrightarrow} (S^2 W)^\vee \cong S^2(W^\vee) \stackrel{\text{disc}}{\longrightarrow} k \, ,
\end{equation}
where the first map is $x\mapsto x\otimes x$ and the map $(S^2 V)^\vee \rightarrow V_{C, \eta}^\vee$ is dual to the defining inclusion.
\par Then $G$ has degree $6$ because the first map has degree $2$, the map $\text{disc}$ has degree $\dim(W)=3$, and all the other maps are linear.

\begin{rem}
    The intuition behind considering the composition
    $$V^\vee \rightarrow S^2(V^\vee) \cong (S^2 V)^\vee \rightarrow V_{C, \eta}^\vee \stackrel{\psi^\vee}{\longrightarrow} (S^2 W)^\vee \cong S^2(W^\vee)$$
    is as follows: we want to understand the image $\Gamma_\eta$ of the map
    $$c: \mathbb{P}^2 \dashrightarrow \mathbb{P}^3\,.$$
    But we do not have access to the map $c$ explicitly; we only have the map $\varphi$ which is linked to the way $c$ transforms quadrics. The best we can do starting from a point $x\in \mathbb{P}^3$ is to form the vector subspace of quadrics vanishing at $x$ and, then, apply the map $\varphi$. This gives us a linear system of conics on $\mathbb{P}^2$ from which we can hope to detect whether or not $x$ is in the image of $c$. Dualizing this consideration shows that it is natural to study the composition
    $$V^\vee \rightarrow S^2(V^\vee) \cong (S^2 V)^\vee \rightarrow V_{C, \eta}^\vee \stackrel{\psi^\vee}{\longrightarrow} (S^2 W)^\vee \cong S^2(W^\vee)\,. $$
\end{rem}
\par The next proposition will show that the homogeneous form $G$ is the square of a cubic form cutting out the cubic symmetroid.
\begin{prop} \label{prop:cubic_squared}
    One has $\mathcal{V}(G)=2 \Gamma_\eta$.
\end{prop}
\begin{proof}
Assume first that $(C, \eta)$ is not bielliptic. We begin by showing that $G$ vanishes on $\Gamma_\eta$. Let $U\subset \Gamma_\eta$ be the complement of the base locus of the birational map $c^{-1}: \Gamma_\eta \dashrightarrow \mathbb{P}^2$. We will show that $G$ vanishes on the open subset $U$. To that end, let $x\in U$ be arbitrary and let $y=c^{-1}(x) \in \mathbb{P}^2$ denote its image. Then $x$ (resp., $y$) gives a non-zero vector in $V^\vee$ (resp., $W^\vee$). We claim that under the composition
$$
\gamma: V^\vee \rightarrow S^2(V^\vee) \cong (S^2 V)^\vee \rightarrow V_{C, \eta}^\vee \stackrel{\psi^\vee}{\longrightarrow} (S^2 W)^\vee \cong S^2(W^\vee) \, ,
$$
$x$ maps to a multiple of $y\otimes y$.
\par Since $W_\eta$ is the image of $\psi$, the map $\gamma$ factors as
$$
V^\vee \rightarrow S^2(V^\vee) \cong (S^2 V)^\vee \rightarrow W_{\eta}^\vee \stackrel{\psi^\vee}{\longrightarrow} (S^2 W)^\vee \cong S^2(W^\vee)\,.
$$
Consider now any quadratic form $q\in S^2 W$ vanishing at $y$. We denote by $\langle \cdot, \cdot \rangle_W$  the natural pairing between $S^2 W$ and $S^2(W^\vee)$ and define $\langle \cdot, \cdot \rangle_V$ similarly. Then we see that
\begin{equation}\label{eqn: pairing}
\langle q, \gamma(x) \rangle_W = \langle q, \psi^\vee(x\otimes x) \rangle_W =\langle \psi(q), x\otimes x\rangle_V \,.
\end{equation}
Next, we know that the map $\psi$ is given by pull-pushing quadrics through the diagram
$$
\begin{tikzcd}
& S\arrow{dr}\arrow{dl}\\
\mathbb{P}(\HH^0(C, \Omega_C (\eta)))\cong \mathbb{P}^2 \arrow[rr, dashrightarrow, "c"]&&\Gamma_\eta \hspace{0.5em},
\end{tikzcd}
$$
i.e., the quadratic form $\psi(q)$ is the unique (up to scalar) quadratic form vanishing on the curve $c(\mathcal{V}(q))\subset \Gamma_\eta$.
\par Since $y=c^{-1}(x)$ and $q$ vanishes at $y$, this implies that $\psi(q)$ vanishes at $x$, i.e.,
$$\langle \psi(q), x\otimes x\rangle_V=0\,.$$
By \ref{eqn: pairing} $\gamma(x)$ must therefore be orthogonal to all the quadratic forms vanishing at $y$ and thus is a multiple of $y\otimes y$. (It could a priori be zero.) This proves the claim.
\par Next, the observation $\text{disc}(y\otimes y)=0$ shows that $G$ vanishes on $\Gamma_\eta$. To prove that $G$ vanishes with multiplicity two, one can use the same argument with adjugate matrices as in the discussion preceding Theorem 4.1.4 in \cite{Dolgachev} to show that $\gamma_\eta^3 \mid G^2 $ where $\gamma_\eta$ is a cubic form cutting out $\Gamma_\eta$. This implies that $\gamma_\eta^2 \mid G$ because $\Gamma_\eta$ is irreducible.
\par It remains to show that $G\neq 0$. One possible argument would be to go through the classification of symmetroid cubic surfaces on \cite[\S1, p. 33]{Catanese} and explicitly verify that $G \neq 0$ in every case. Alternatively, we propose the following general proof. Choose a plane $H\subset \mathbb{P}^3$ such that $\Gamma_\eta \cap H$ is smooth. Then the restriction $\gamma_{\eta_{|H}}$ gives a determinantal equation for the smooth plane curve $\Gamma_\eta \cap H$. Hence, the theory of determinantal equations from Dolgachev’s book \cite[Section 4.1.2]{Dolgachev} applies. The degree six form $G_{|H}$ appears as a special case of the construction preceding \cite[Theorem 4.1.4]{Dolgachev} (where it is denoted by $\det(N)$). In loc.\ cit. he proves that this polynomial is non-zero. Therefore, we conclude that $G_{|H}\neq 0$ and thus $G\neq 0$. This proves the non-bielliptic case.
\par Assume now that $(C, \eta)$ is bielliptic. Then $\Gamma_\eta$ is a cone over a smooth plane cubic $E\subset \mathbb{P}^2$. Furthermore, after applying a coordinate transformation that maps the vertex to $(0:0:0:1)$ it is easy to see that $\mathcal{V}(G)$ is a cone over the same vertex by looking at the resulting equation. (This follows from the fact that in the bielliptic case, $\psi$ maps a conic to the cone over it, as mentioned in Section \ref{sec:bielliptic}.) Thus we are reduced to a statement about the symmetric determinantal representation of $E$ in $\mathbb{P}^2$. Then Dolgachev's argument applies directly and the proposition follows.
\end{proof}
In order to apply the proposition we need to know the map $\psi$. So far we have explained how to reconstruct the map $\varphi$. Now, the map $\psi$ is characterized by being the right inverse of $\varphi$ with image $W_\eta$. Therefore, it suffices to recover the linear subspace $W_\eta \subset S^2V$. For that purpose, we have the following lemma:
\begin{lem}\label{lem:invtrick}
    Let $W_\eta \subset S^2V$ be as before. Then the orthogonal complement $W_\eta^\perp \subset S^2V^\vee$ is closed under the binary operation
    $$S^2V^\vee\times S^2V^\vee \rightarrow S^2 V^\vee:\, (Q_1, Q_2)\mapsto (Q_1^\sharp+Q_2^\sharp)^\sharp$$
    where $(\cdot)^\sharp:S^2V^\vee \rightarrow S^2 V$ corresponds to the map that sends a quadric to its dual.  (In coordinates, the map $(\cdot)^\sharp$ is given by the cofactor matrix.)
\end{lem}
\begin{proof}
    The cubic symmetroids that satisfy the conclusion of the lemma form a closed subset of the parameter space of all symmetroids. Therefore, it suffices to consider the case where $\Gamma_\eta$ is a Cayley cubic. We may choose coordinates such that $\Gamma_\eta$ is given by $x_0x_1x_2x_3 \sum_{i=0}^3 \frac{1}{x_i}$. In this case $W_\eta^\perp$ consists of those quadratic forms whose Gram matrix is diagonal. Since $(\cdot)^\sharp$ is given by taking the cofactor matrix, the lemma follows.
\end{proof}
\begin{rem}
    Strictly speaking, it is more natural to view $(\cdot)^\sharp$ as a map from $S^2V^\vee$ to $S^2V \otimes (\bigwedge^4V^\vee)^{\otimes 2}$. However, since $(\bigwedge^4V^\vee)$ is one-dimensional, this is not important for the lemma.
\end{rem}

\section{Explicit formulas} \label{subsec:main}

\subsection{The formulas} \label{subsec:formulas}
We now give explicit formulas for the equations of a generic genus four curve in terms of its theta constants. We begin with a summary of our method, and then give more detailed explanations of the individual steps.

\begin{overview}\label{alg}
Starting with the theta constants $\thetacha{C}{\delta}{\varepsilon}$ of a generic smooth genus $4$ curve $C$ defined over $k$, we proceed as follows.
\begin{enumerate}
\item Choose a 2-torsion point $\eta\in \Jac(C)[2]$. Pick a $2$-level structure so that $\eta$ is identified with $\frac{1}{2}\begin{bmatrix}
 0 & 0 &0 & 0\\
 1 & 0 &0 & 0
\end{bmatrix}$.\\
\item Compute the tritangent planes $H_i, H_i'$ from the Equations (\ref{eqn:tritangents}).
\item Compute the theta constants for the plane quartic $X$ via the Schottky--Jung relations
$$\thetacha{X}{\delta}{\varepsilon}^2=\thetacha{C}{0 & \delta}{0 & \varepsilon} \thetacha{C}{0 & \delta}{1 & \varepsilon}\,. $$
\item Compute the 10 bitangent lines $\ell_i$ via the Aronhold-Weber formulas (\Cref{subsec:aronholdweber}) from the theta constants $\thetacha{X}{\delta}{\varepsilon}$.
\item From the Equations (\ref{eqn:Milne10})
$$\varphi(H_i H_i')=\lambda_i \ell_i^2$$
we get a homogeneous linear system of equations that we can solve for the $\lambda_i$.\\
\item From the $\lambda_i$ compute a matrix for the linear map $\varphi$.\\
\item Compute a generator for $\ker(\varphi)$; call it $Q$.\\
\item Use Lemma \ref{lem:invtrick} to compute $W_\eta$.\\
\item We then compute $\psi$ as the unique right inverse of $\varphi$ with image $W_\eta$.\\
\item From $\psi$ compute the degree six polynomial $G$ as defined in the beginning of Section \ref{sec:cubic}. Compute a square root $\Gamma_\eta$ of $G$.\\
\item We obtain the desired curve $C$ as the intersection of $Q$ and $\Gamma_\eta$.

\end{enumerate}
\end{overview}

We will now make the various steps explicit, beginning with (2). As in \Cref{subsec:aronholdweber} we use the convention that
$
     \thetacha{C}{\delta_0 & \delta_1 & \delta_2 & \delta_3 }{\varepsilon_0 & \varepsilon_1 & \varepsilon_2 & \varepsilon_3} = \vartheta_{C,i}$, where 
$$i
  = 2(\varepsilon_3 + 2\varepsilon_2 + 4\varepsilon_{1} + 8\varepsilon_0)
  + 2^{5}(\delta_3 + 2 \delta_2 + 4\delta_1 + 8\delta_0).
$$

Given the two-torsion point
$\eta=\frac{1}{2}\begin{bmatrix}
    0 & 0 & 0 & 0\\
    1 & 0 & 0 & 0
\end{bmatrix}$,
recall that in the proof of \Cref{lem:special_characteristics} we chose the following theta characteristics of the genus 4 curve $C$ that we want to reconstruct.

\begin{align*}
\xi_1 &=\frac{1}{2}\begin{bmatrix} 1 & 1 & 1 & 0 \\ 1 & 1 &  1 & 0 \end{bmatrix}, \quad
\xi_2 =\frac{1}{2}\begin{bmatrix} 1 & 0 & 1 & 0 \\ 0 & 0 &  1 & 0 \end{bmatrix}, \quad
\xi_3 =\frac{1}{2}\begin{bmatrix} 1 & 1 & 1 & 0 \\ 0 & 0 &  1 & 0 \end{bmatrix},\\
\xi_4 &=\frac{1}{2}\begin{bmatrix} 1 & 0 & 1 & 0 \\ 0 & 1 &  1 & 0 \end{bmatrix}, \quad
\xi_5 =\frac{1}{2}\begin{bmatrix} 0 & 1 & 1 & 0 \\ 0 & 1 &  0 & 0 \end{bmatrix},
\end{align*}
and the $\chi_i$:
\begin{align*}
\chi_1 &=\frac{1}{2}\begin{bmatrix} 0 & 1 & 1 & 1 \\ 0 & 1 &  1 & 1 \end{bmatrix}, \quad
\chi_2 =\frac{1}{2}\begin{bmatrix} 0 & 1 & 0 & 1 \\ 0 & 1 &  0 & 0 \end{bmatrix},\quad
\chi_3 =\frac{1}{2}\begin{bmatrix} 0 & 1 & 0 & 0 \\ 0 & 1 &  0 & 1 \end{bmatrix},\\
\chi_4 &=\frac{1}{2}\begin{bmatrix} 0 & 1 & 0 & 0 \\ 0 & 1 &  1 & 0 \end{bmatrix},\quad
\chi_5 =\frac{1}{2}\begin{bmatrix} 0 & 1 & 1 & 0 \\ 0 & 1 &  0 & 0 \end{bmatrix}, \quad 
\chi_6 =  \frac{1}{2}\begin{bmatrix} 0 & 1 & 1 & 0 \\ 0 & 1 &  0 & 1 \end{bmatrix},\\
\chi_7 &=\frac{1}{2}\begin{bmatrix} 0 & 1 & 0 & 0 \\ 0 & 1 &  1 & 1 \end{bmatrix}, \quad
\chi_8 =\frac{1}{2}\begin{bmatrix} 0 & 1 & 0 & 1 \\ 0 & 1 &  1 & 0 \end{bmatrix}, \quad
\chi_9 =\frac{1}{2}\begin{bmatrix} 0 & 1 & 1 & 1 \\ 0 & 1 &  0 & 0 \end{bmatrix}, \\
\chi_{10} &=\frac{1}{2}\begin{bmatrix} 0 & 1 & 0 & 0 \\ 0 & 1 &  0 & 0 \end{bmatrix}
\end{align*}
and $\chi_i' = \chi_i +\eta$ for $i \in \{1,\ldots,10\}$.

The tritangent planes $H_i, H_i'$ correspond to the odd theta characteristics $\chi_i, \chi_i'$. We have shown in \Cref{sub:oddtheta} that for any $i\in \{1, \ldots, 10\}$ the tritangent plane $H_i$ is given by the equation

\begin{equation}
\label{eqn:tritangents2}
\begin{split}
\frac{D(\chi_i, \xi_2, \xi_3, \xi_4)  }{D(\xi_5, \xi_2, \xi_3, \xi_4)} x_0+&\frac{D( \xi_1, \chi_i,\xi_3, \xi_4)  }{D( \xi_1, \xi_5, \xi_3, \xi_4)} x_1+\\ &\frac{D( \xi_1,\xi_2, \chi_i, \xi_4)  }{D( \xi_1,\xi_2, \xi_5, \xi_4)} x_2+
\frac{D( \xi_1,\xi_2,\xi_3, \chi_i)  }{D( \xi_1,\xi_2,\xi_3, \xi_5)} x_3=0.
\end{split}
\end{equation}
The analogous formula with $\chi_i$ replaced by $\chi_i'$ gives the tritangent planes $H_i'$.

All $D(\alpha, \beta, \gamma, \delta)$ can be explicitly written in terms of theta constants (\Cref{lem:special_characteristics}, \Cref{ThmJacobi}). One, for example, has
\begin{equation}
    \begin{aligned}
        D(\xi_5, \xi_2, \xi_3, \xi_4) &= - \vartheta_{C,215}  \vartheta_{C,125}\vartheta_{C, 213}\vartheta_{C,126}\vartheta_{C,236}\vartheta_{C, 111}\\
        & \qquad + \vartheta_{C,85}\vartheta_{C,255}\vartheta_{C,87}\vartheta_{C,252}\vartheta_{C,110}\vartheta_{C,237}
    \end{aligned}
\end{equation}
and the other expressions are similar.

We can now write
 $$H_i=\sum_{j=0}^3 h_{i,j} x_j,\, H_i'=\sum_{j=0}^3 h_{i,j}' x_j,\quad i=1,\ldots, 10.$$
Next we proceed with step (3). To every pair of theta characteristics 

$$\chi_i =\frac{1}{2}\begin{bmatrix}
    0 & \delta_1 & \delta_2 & \delta_3\\
    0 & \varepsilon_1 & \varepsilon_2 & \varepsilon_3
\end{bmatrix}, \quad \chi_i' =\frac{1}{2}\begin{bmatrix}
    0 & \delta_1 & \delta_2 & \delta_3\\
    1 & \varepsilon_1 & \varepsilon_2 & \varepsilon_3
\end{bmatrix}$$ we can associate a theta characteristic $\zeta_i =\frac{1}{2}\begin{bmatrix}
    \delta_1 & \delta_2 & \delta_3\\
    \varepsilon_1 & \varepsilon_2 & \varepsilon_3
\end{bmatrix}$ for the genus 3 curve $X$. Each of the $\zeta_i$ has an associated bitangent line \begin{equation}\label{eqn:bitanlines}
\ell_i=\sum_{i=0}^2 b_{i,j} y_j
\end{equation}
that we can compute via the Aronhold--Weber formulas from the theta constants $\thetacha{X}{\delta}{\varepsilon}$ (see \Cref{subsec:aronholdweber}). 
Here we use that \cite[Theorem 1]{FarkasRauch}
$$\thetacha{X}{\delta}{\varepsilon}=\sqrt{\thetacha{C}{0 & \delta}{0 & \varepsilon} \thetacha{C}{0 & \delta}{1 & \varepsilon}}\,. $$

Using the same notation as in \Cref{eq:explicitbitangents}, the ten bitangents are then given by
\begin{equation}
\label{eq:tenbitangents}
\begin{aligned}
  &\ell_1=y_0,\quad \ell_2=y_0+y_1+y_2,\quad \ell_3=a_{10}y_0+a_{11}y_1+a_{12}y_2,\\
  &\ell_4=u_1,\quad \ell_5=u_2,\quad \ell_6=y_0+y_1+u_2, \quad  \ell_7=y_0+y_2+u_1,\quad\\
  &\ell_8=\frac{u_1}{a_{11}} + k_1 (a_{10}y_0 + a_{12}y_2),\quad
\ell_9=\frac{u_2}{a_{12}} + k_1 (a_{10}y_0 + a_{11}y_1),\\
&\ell_{10}=\frac{u_0}{1-k_1 a_{11}a_{12}} + \frac{u_1}{1-k_1 a_{10}a_{12}} +\frac{u_2}{1-k_1 a_{10}a_{11}}\,.
    \end{aligned}
\end{equation}

In step (4) we want to get a homogeneous linear system of equations for the $\lambda_i$ from the Equations (\ref{eqn:Milne10})
$\varphi(H_i H_i')=\lambda_i \ell_i^2$. For this purpose, we first find the linear relations between the $H_iH_i'$. Thus, we form the matrix $H^\square$ whose columns are the entries of the symmetric matrices corresponding to $H_i H_i',\, i=1, \ldots, 10$, and compute its kernel. We have
$$  H^\square := \frac{1}{2}\begin{pmatrix}
     2 h_{1,0}h_{1,0}' &\quad \cdots\quad & 2 h_{10,0}h_{10,0}'\\
     h_{1,0} h_{1,1}'+h_{1,1} h_{1,0}' &\, \cdots\, &h_{10,0} h_{10,1}'+h_{10,1} h_{10,0}'\\
     h_{1,0} h_{1,2}'+h_{1,2} h_{1,0}'&\, \cdots\,   &h_{10,0} h_{10,2}'+h_{10,2} h_{10,0}'\\
     h_{1,0} h_{1,3}'+h_{1,3} h_{1,0}' & \, \cdots\,  &h_{10,0} h_{10,3}'+h_{10,3} h_{10,0}'\\
     2 h_{1,1} h_{1,1}' & \, \cdots\,   &2 h_{10,1} h_{10,1}'\\
     h_{1,1} h_{1,2}'+h_{1,2} h_{1,1}' &\, \cdots\, & h_{10,1} h_{10,2}'+h_{10,2} h_{10,1}'\\
     h_{1,1} h_{1,3}'+h_{1,3} h_{1,1}' &\, \cdots\,  &h_{10,1} h_{10,3}'+h_{10,3} h_{10,1}'\\
     2 h_{1,2} h_{1,2}'& \, \cdots\,   & 2 h_{10,2} h_{10,2}'\\
     h_{1,2} h_{1,3}'+h_{1,3} h_{1,2}'&\, \cdots\,  & h_{10,2} h_{10,3}'+h_{10,3} h_{10,2}'\\
     2 h_{1,3} h_{1,3}'&\, \cdots\,  & 2 h_{10,3} h_{10,3}'&

\end{pmatrix}
$$
and the kernel has dimension 3, as we saw in Section \ref{subsec:2linear} that the span of the $H_iH_i'$ has dimension 7.
Write 
$$
\lspan(n_1, n_2, n_3):= \ker(H^\square) \, .
$$
Next, we write down the matrix $L$ whose columns are the entries of the symmetric matrices corresponding to the $\ell_i^2$ 
$$  L := \begin{pmatrix}
     b_{1,0}^2 & \quad \cdots \quad  & b_{10,0}^2\\
     b_{1,0} b_{1,1}& \cdots  & b_{10,0} b_{10,1}\\
     b_{1,0} b_{1,2}& \cdots  &b_{10,0} b_{10,2}\\
     b_{1,1}^2 & \cdots  &b_{10,1}^2\\
     b_{1,1} b_{1,2} &\cdots & b_{10,1} b_{10,2}\\
     b_{1,2}^2& \cdots  & b_{10,2}^2\\
\end{pmatrix}
$$
where the $b_{i,j}$ are as in \ref{eqn:bitanlines}.
The $\lambda_i$ are then found by solving the linear system of 
equations
$$ \lspan(\lambda):= \ker\begin{pmatrix}
    L \diag(n_1)\\
    L \diag(n_2)\\
    L \diag(n_3)
\end{pmatrix} $$
because by linearity, any relation between the $H_iH_i'$ is sent to zero under $\varphi$. Thus the condition $\varphi(H_i H_i')=\lambda_i \ell_i^2$ leads to the linear system as above.

(5): From the $\lambda_i$ compute a matrix for the linear map $\varphi$ using the basis $H_1 H_1',\ldots, H_{7} H_7'$ on $V_{C, \eta}$. (Generically, the $H_iH_i'$ form a basis; if not, it is clear what one needs to change).\\
$$M_\varphi := L \begin{pmatrix}
    \lambda_1 & 0 & 0 &0 &0 & 0 & 0 \\
    0 & \lambda_2 & 0 & 0 &0 & 0 & 0\\
    0 & 0 & \lambda_3 & 0 & 0 &0 & 0\\
    0 & 0 & 0 & \lambda_4 & 0 & 0 &0\\
    0 & 0 & 0 & 0 & \lambda_5 & 0 & 0\\
    0 & 0 & 0 &  0 &0 & \lambda_6 & 0\\
    0 & 0 &0 &0 &0 &0 & \lambda_7\\
    0 & 0 &0 &0 &0 &0 &0\\
    0 & 0 &0 &0 &0 &0 &0\\
    0 & 0 &0 &0 &0 &0 &0\\
\end{pmatrix}$$

(6): A generator for $\ker(\varphi)$ will give us a quadratic form $Q$ vanishing on the curve $C$. Indeed, if we write $\lspan(v):= \ker(M_\varphi)$, then an equation of $Q$ will be given by
$$Q := \sum_{i=1}^7 v_i H_i H_i'.$$

(7): We will now use the matrix $M_\varphi$ we computed to find a Cayley cubic $\Gamma_\eta$ cutting out the curve $C$. For this, we describe how to make Lemma \ref{lem:invtrick} explicit. We first compute a basis for $V_{C, \eta}^\perp \subset S^2V^\vee$, that is
$$\lspan(c_1, c_2, c_3):= \ker((H^\square)^t),$$ and write down the matrices $Q_1, Q_2$ corresponding to $c_1$ and $c_2$:

$$ Q_1:=\begin{pmatrix}
    2c_{1,1} & c_{1,2} & c_{1,3} & c_{1,4}\\
    c_{1,2} & 2c_{1,5} & c_{1,6} & c_{1,7}\\
    c_{1,3} &c_{1,6}& 2c_{1,8} & c_{1,9}\\
    c_{1,4} &c_{1,7}&c_{1,9}& 2c_{1,10}
\end{pmatrix},\,  Q_2:=\begin{pmatrix}
    2c_{2,1} & c_{2,2} & c_{2,3} & c_{2,4}\\
    c_{2,2} & 2c_{2,5} & c_{2,6} & c_{2,7}\\
    c_{2,3} &c_{2,6}& 2c_{2,8} & c_{2,9}\\
    c_{2,4} &c_{2,7}&c_{2,9}& 2c_{2,10}
\end{pmatrix}.$$
Then, by Lemma \ref{lem:invtrick}, the symmetric matrix
$$Q':=(Q_1^{-1}+Q_2^{-1})^{-1}$$
defines an element of $W_\eta^\perp$, which for a generic curve will generate $W_\eta^\perp$ together with $V_{C, \eta}^\perp$.
\par (8): Now we can compute a matrix for $\psi$ by noticing that it is characterized by being a right inverse of $\varphi$ and having image $W_\eta$:
$$q':=\begin{pmatrix} Q'_{11} &Q'_{12} &Q'_{13} &Q'_{14} &Q'_{22} &Q'_{23} &Q'_{33} &Q'_{34} &Q'_{44} \end{pmatrix}$$
$$M_{\psi} := \begin{pmatrix}
    M_\varphi\\
    q' H^\square
\end{pmatrix}^{-1}\begin{pmatrix}
    \mathsf{1}_6\\
    0_{1\times6}
\end{pmatrix}$$
(9): Finally, we can use $M_\psi$ to compute the degree six polynomial $G$ from the beginning of Section \ref{sec:cubic}. Let $d$ be the form given as an element of $\Mat_{6,1}(k[x_0, x_1, x_2, x_3])$ by
$$d := M_\psi \begin{pmatrix}
    H_1 H_1' & H_2 H_2' & H_3 H_3' & H_4 H_4' & H_5 H_5' & H_6 H_6' & H_7 H_7'
\end{pmatrix}^t,$$
and write $\Delta$ for its associated matrix:
$$\Delta := \begin{pmatrix}
d_1 & d_2 & d_3\\
d_2 & d_4 & d_5\\
d_3 & d_5 & d_6
\end{pmatrix}.$$  
Now $G$ is the degree 6 form given as 
$G=\det(\Delta)$ and we find the cubic $$\Gamma_\eta=\sqrt{G}$$ which, by construction, will be equal to zero on $C$. Now we have explicitly written down $Q$ and $\Gamma_\eta$ such that
$$Q \cap \Gamma_\eta = C.$$

Collecting the formulas given above, we have proven the following theorem.
\begin{thm} \label{thm:main_result}
Let $C$ be a generic genus $4$ curve over an algebraically closed field $k$ of characteristic 0 or of characteristic $p$ with $p$ large enough. The explicit formulas above give the equations of $C$ as an intersection of a quadric and a Cayley cubic starting from the theta constants. They use only the following algebraic operations: Elementary arithmetic, taking square roots, and solving linear systems of equations.
\end{thm}

\begin{rem} These formulas can be seen as an analogue of the result of Aronhold--Weber in genus 4. We similarly choose a suitable system of odd theta characteristics and describe the coefficients of the corresponding bitangents and tritangents involved in terms of the theta constants. To reconstruct the genus 3 curve, the method of Aronhold--Weber requires you to solve the linear system of equations given by \Cref{eq:aronholdweber}. The computations needed in order to reconstruct the genus 4 curve are slightly more involved, but at their core also only consist of solving linear systems of equations (inverting matrices), and elementary algebraic operations.
\end{rem}

\subsection{Discussion} \label{subsec:discussion}

\begin{rem}
When using the Schottky--Jung relations one takes a square root to compute the $\thetacha{X}{\delta}{\varepsilon}$. This leads to a sign choice that has to be made correctly. For this one can use the discussion following \cite[Theorem 3.1]{Glass} where Glass makes use of the quartic Riemann theta 
formulas to find a correct choice of signs. At this point, the assumption that the $\thetacha{C}{\delta}{\varepsilon}$ are the theta constants of a genus $4$ curve is used. Indeed, the latter implies that the Schottky modular form vanishes. By construction of this modular form, this is equivalent to requiring that the $\thetacha{X}{\delta}{\varepsilon}$ satisfy the quartic Riemann theta formulas.
\par When the theta constants are given in an exact field, then taking square roots might lead to an auxilliary field extension. This field extension only appears in intermediate steps. The final result of Algorithm \ref{alg} consists of equations defined over the same field as the theta constants.
\end{rem}

\begin{rem}
If one of the $\thetacha{C}{0 & \delta}{0 & \varepsilon}$ or $\thetacha{C}{0 & \delta}{1 & \varepsilon}$ vanish, then the Schottky--Jung relations imply that $X$ has a vanishing even theta-null. This means that the curve $X$ becomes hyperelliptic. Our method still works in this case by making the following adjustment: the canonical map $X\rightarrow \mathbb{P}^2$ is a degree two cover of a smooth conic ramified in eight points. Replace the 28 bitangent lines by the $\binom{8}{2}=28$ lines through pairs of ramification points (see also Remark \ref{rem: MilneHyp}). 
\par Notice that in this situation, the curve $C$ has a vanishing theta null and thus the canonical quadric $Q$ becomes singular. Conversely, if $C$ has a vanishing theta null, we can achieve that $\thetacha{C}{0}{0}=0$ by choosing the level structure appropriately. Thus we can arrange for the curve $X$ to be hyperelliptic.
\end{rem}
\begin{rem}\label{rem: fail}
There are three possibilities for the curve $C$ to not be generic enough:
\begin{enumerate*}[(i)]
    \item there is a division by zero in step (2),
    \item the products $H_i H_i'$ do not generate $V_{C, \eta}$, or
    \item the $\lambda_i$ are not uniquely determined by the linear system of equations in step (5). 
\end{enumerate*}
The first situation can occur; see Remark \ref{rem:Dzer}. However, the authors do not know if there exist curves exhibiting either of the other two phenomena. Using the same strategy as in the proof of Lemma \ref{lem:gen_set}, \ref{lem:proof_example}, i.e., by producing an example over a finite field, we obtained that the open locus in the moduli space given by the negations of conditions (ii) and (iii) has non-empty intersection with the following loci:
\begin{itemize}
    \item $C$ has a vanishing even theta constant, i.e., the quadric $Q$ is singular;
    \item Any degeneration of the cubic symmetroid $\Gamma_\eta$;
    \item In particular, the bielliptic locus.
\end{itemize}
\par Furthermore, if the reconstruction fails for one of the three reasons above, all hope is not lost---one can choose a different $2$-torsion point $\eta$ and try again. As there are 255 possibilities, and every $2$-torsion point will only give a closed locus (of expected codimension 1) parametrizing curves for which the construction can go wrong, we conjecture that for every non-hyperelliptic genus $4$ curve over a field of characteristic $\neq 2$ there exists a two-torsion point $\eta\in \Jac(C)[2]\setminus\{0\}$ such that our reconstruction method works. 
\end{rem}

\section{Examples} \label{sec:examples}

    Although our main result Theorem \ref{thm:main_result} requires the base field $k$ to be algebraically closed, with additional information one can recover equations over arithmetic base fields. Over finite fields our method can be used to recover equations of genus $4$ curves from the knowledge of the algebraic theta constants. These are defined over a field extension of manageable size as the degree of this extension can be bounded using the fact that extensions of finite fields have cyclic Galois group. (See \cite{CelikKulkarni} for examples of this type.)

In stark contrast, in the number field case the field of definition of the algebraic theta constants is typically a huge field extension of degree up to 
    $$
    [\Sp_{2g}(\mathbb{Z}): \Gamma(4,8)] = 2^{g^2+g(2g+1)+2g} \prod_{i=1}^g (2^{2i}-1)
    $$
    (cf., \cite[Chapter V, Theorem 4]{IgusaThetaBook}).
    It is therefore imperative to use floating point approximations to the theta constants. These can be computed analytically provided a small period matrix is known.
    
    In order to retrieve equations over a number field $K$, we need more information. If in addition, we are also given a big ($g \times 2g$) period matrix for $C$ with respect to a $K$-rational basis for $\HH^0(C, \Omega_C)$, we can recover equations for $C$ defined over $K$ as follows. First applying our theorem, we obtain approximate equations for $C$ with coefficients in $\CC$. Denote by $\widetilde{C}$ this complex curve, which is isomorphic to $C$ over $\CC$. 
    Computing the big period matrix for $\widetilde{C}$ and comparing it with that of $C$, we can determine the $\operatorname{PGL}_4$ transformation that changes coordinates on $\PP^3$ to make the equations of $\widetilde{C}$ defined over $K$. We then use LLL to recognize the resulting coefficients as elements of $K$. See Examples \ref{ex:gluing} and \ref{ex:modular} below for concrete applications of this process.

We now discuss a few examples. First we will construct an example over $\mathbb{F}_{37}$ and evaluate the formulas from \Cref{subsec:main} with the aid of \Magma{}. (As we construct this example ourselves, and we know the equations of the curve already, we do not need to perform steps (1) - (3) to find equations for the $\ell_i$ and the $H_i$.) This will on one hand serve to illustrate how the construction works in practice and on the other hand complete the proof for Lemmas \ref{lem:gen_set} and \ref{lem:proof_example}. (The \Magma{} computations can be found in the files \texttt{Finite-Field-example.m} and \texttt{construct-Finite-Field-example.m} in \cite{GitHub}.)

\begin{example}\label{ex:finite_field}
\vspace{1em}

\par  Consider the plane quartic $X\subset \mathbb{P}^2_{\mathbb{F}_{37}}$ with equation
\begin{align*}
X: t^4 &+ 14 t^3 u + 16 t^3 v + 32 t^2 u^2 + 26 t^2 u v + 18 t^2 v^2 + 29 t u^3 + 4 t u^2 v + 11 t u v^2\\
    & + 2 t v^3 + 26 u^4 + 16 u^3 v + 27 u^2 v^2 + 22 u v^3 + 11 v^4 = 0.
\end{align*}
The curve $X$ was chosen in such a way that all bitangents are rational. This was achieved using the Aronhold--Weber formulas and as a consequence, we explicitly know equations for all bitangents.
\par Next, using Recillas's trigonal construction \cite[Theorem 2.15]{Recillas} with the explicit geometric version of \cite[Theorem 1.5]{BruinSertoz} one can compute a genus 4 curve $C$ with a two-torsion point $\eta$ such that $\Jac(X)$ is isomorphic to the Prym. All the explicit properties of the curve $C$ that we claim below will be a consequence of this construction.
\par We find that the canonical embedding of the curve $C$ into $\mathbb{P}^3$ is cut out by the equations
\begin{align*}
        0 &= 11xw + yz,\\
        0 &= x^3 + 3x^2y + 9x^2z + 15x^2w + 23xy^2 + 10xyw + 12xz^2 + 6xzw + 12xw^2\\
            & \qquad + 8y^3 + y^2w + 17yw^2 + 11z^3 + 14z^2w + 7zw^2 + 28w^3 \, .
\end{align*}

We now start performing step (4). The curve 
$C$ has the following 10 products of pairs of tritangent planes $(H_i, H_i')$ differing by $\eta$. (Note that many of these are irreducible quadrics over $\F_{37}$, but factor as a product of two linear forms over $\F_{37^2}$.)
\begin{small}
\begin{align*}
H_1H_1'&= 2x^2 + 9xy + 18xz + 6xw + 8y^2 + 20yz + 5yw + 15z^2 + 23zw + 33w^2,\\    
H_2H_2'&= 21x^2 + 25xy + 20xz + 29xw + 2y^2 + 32yz + 34yw + 16z^2 + 5zw + 36w^2,\\
H_3H_3'&= 4x^2 + 13xy + 28xz + 9xw + 7y^2 + 28yz + 22yw + 19z^2 + 20zw + 33w^2,\\
H_4H_4'&= 30x^2 + 27xy + 6xz + 29xw + 28y^2 + 4yz + 16yw + 25z^2 + 4zw + 13w^2,\\
H_5H_5'&= 24x^2 + 22xy + 2xz + 32xw + 22y^2 + 4yz + 27yw + 31zw + 21w^2,\\
H_6H_6'&= 9x^2 + 16xy + 19xz + 6xw + 17y^2 + 12yz + 14yw + 23z^2 + 15zw + 32w^2, \\
H_7H_7'&= 29x^2 + 26xy + 30xz + 36xw + 22y^2 + 34yz + 11yw + 12z^2 + 35zw + 23w^2, \\
H_8H_8'&= 25x^2 + 26xy + 30xz + 24xw + 32y^2 + 16yz + 2yw + 28z^2 + 3zw + 33w^2, \\
H_9H_9'&= 8x^2 + 30xz + 17xw + 19y^2 + 31yz + 3yw + 23z^2 + 17w^2, \\
H_{10}H_{10}'&= 13x^2 + 34xy + 28xz + 14y^2 + 13yz + 22yw + 29z^2 + 16zw + 14w^2\, .
\end{align*}
\end{small}
(It is important to ensure that the odd theta characteristics corresponding to $(H_i, H_i')$ are given by the $(\chi_i, \chi_i')$ from Lemma \ref{lem:special_characteristics}.)
\par By explicit computation one checks that the 7 quadratic forms $H_3 H_3', \ldots, H_{10} H_{10}'$ are linearly independent and thus form a basis of $V_{C, \eta}$. This proves Lemma \ref{lem:gen_set}.
\par Next, the bitangent lines of $X$ corresponding to $(H_i, H_i')$ under Milne’s bijection are given by 
\[
\begin{array}{lll}
\ell_1 = 12t + 16u + 23v, & \ell_2 = 3t + 19u + 10v, & \ell_3 = t + u + v, \\
\ell_4 = 10t + 35u + 23v, & \ell_5 = 7t + 17u + 24v, & \ell_6 = 23t + 27u + 30v, \\
\ell_7 = 24t + 27u + 31v, & \ell_8 = t, & \ell_9 = 35t + 19u + v, \\
\ell_{10} = 13t + 17u + 23v \, . & & 
\end{array}
\]
To recover $\varphi$, we first write down the matrices $H^\square$ and $L$.
\begin{small}
$$
H^\square = \begin{pmatrix} 2 & 23 & 9 & 3 & 8 & 10 & 21 & 15 & 30 & 33 \\
21 & 31 & 10 & 33 & 2 & 16 & 17 & 16 & 21 & 36 \\
14 & 25 & 14 & 23 & 7 & 14 & 11 & 19 & 10 & 33 \\
30 & 32 & 3 & 33 & 28 & 2 & 8 & 25 & 2 & 13 \\
24 & 11 & 1 & 16 & 22 & 2 & 32 & 0 & 34 & 21 \\
 9 & 8 & 28 & 3 & 17 & 6 & 7 & 23 & 26 & 32 \\
29 & 13 & 15 & 18 & 22 & 17 & 24 & 12 & 36 & 23 \\
25 & 13 & 15 & 12 & 32 & 8 & 1 & 28 & 20 & 33 \\
 8 & 0 & 15 & 27 & 19 & 34 & 20 & 23 & 0 & 17 \\
13 & 17 & 14 & 0 & 14 & 25 & 11 & 29 & 8 & 14 \\
\end{pmatrix}
$$
$$
L = \begin{pmatrix} 
33 & 14 & 34 & 34 & 33 & 11\\
 9 & 3 & 23 & 28 & 10 & 26 \\
 1 & 2 & 2 & 1 & 2 & 1 \\
26 & 34 & 16 & 4 & 19 & 11 \\
12 & 16 & 3 & 30 & 2 & 21 \\
11 & 21 & 11 &26 & 29 & 12 \\
21 & 1 & 8 & 26 & 9 & 36 \\
 1 & 0 & 0 & 0 & 0 & 0 \\
 4 & 35 & 33 &28 & 1 & 1 \\
21 & 35 & 6 & 30 & 5 & 11
\end{pmatrix}
$$
\end{small}
We then compute the kernel of $H^\square$:

\begin{align*}\ker(H^\square) = \lspan(
&( 1,  0,  0,  5, 24,  1,  3, 21,  1, 18)^t,\\
&( 0,  1,  0, 20, 17, 11, 14, 22, 33, 30)^t,\\
&( 0,  0,  1, 26, 18, 34,  7, 29,  8, 25)^t).
\end{align*}
Writing $n_1,n_2,n_3$ for the above basis of $\ker(H^\square)$ we get
\begin{equation*}
    \resizebox{\textwidth}{!}{$\begin{pmatrix}
    L \diag(n_1)\\
    L \diag(n_2)\\
    L \diag(n_3)
\end{pmatrix} =
    \begin{pmatrix}
    33 & 14 & 34 & 34 & 33 & 11 & 0 & 0 & 0 & 0 & 0 & 0 & 0 & 0 & 0 & 0 & 0 & 0\\
    0 & 0 & 0 & 0 & 0 & 0 & 9 & 3 & 23 & 28 & 10 & 26 & 0 & 0 & 0 & 0 & 0 & 0\\
    0 & 0 & 0 & 0 & 0 & 0 & 0 & 0 & 0 & 0 & 0 & 0 & 1 & 2 & 2 & 1 & 2 & 1\\
    19 & 22 & 6 & 20 & 21 & 18 & 2 & 14 & 24 & 6 & 10 & 35 & 10 & 33 & 9 & 30 & 13 & 27\\
    29 & 14 & 35 & 17 & 11 & 23 & 19 & 13 & 14 & 29 & 34 & 24 & 31 & 29 & 17 & 22 & 36 & 8\\
    11 & 21 & 11 & 26 & 29 & 12 & 10 & 9 & 10 & 27 & 23 & 21 & 4 & 11 & 4 & 33 & 24 & 1\\
    26 & 3 & 24 & 4 & 27 & 34 & 35 & 14 & 1 & 31 & 15 & 23 & 36 & 7 & 19 & 34 & 26 & 30\\
    21 & 0 &  0 &  0 &  0 &  0 & 22 &  0 &  0 &  0 &  0 & 0 & 29 &  0 &  0 &  0 &  0 &  0\\
    4 & 35 & 33 & 28 &  1 &  1 & 21 & 8 & 16 & 36 & 33 & 33 & 32 & 21 &  5 &  2 &  8 &  8\\
    8 &  1 & 34 & 22 & 16 & 13 &  1 & 14 & 32 & 12 &  2 & 34 &  7 & 24 &  2 & 10 & 14 & 16
    \end{pmatrix}
    $}
\end{equation*}
One readily verifies that the left kernel of this matrix is one-dimensional and generated by $$(\lambda_1,\ldots , \lambda_{10}) = (1,21,13,32,8,33,17,3,19,18).$$
This suffices to prove Lemma \ref{lem:proof_example} and finishes step (4).

(5): The matrix $M_{\varphi}$ for
$$\varphi: V_{C, \eta}\longrightarrow S^2 \HH^0(C, \Omega_C(\eta))$$
is given by
$$M_{\varphi} = \begin{pmatrix}
18 & 22 & 30 & 24 & 3 &  2 &  8 \\
15 & 17 & 27 & 17 &  0 & 36 & 1 \\
31 & 24 & 30 & 25 &  0 & 35 & 34 \\
17 & 18 &  7 & 35 &  0 & 14 & 22 \\
16 & 16 & 32 &  5 & 0 & 19 & 16 \\
19 & 20 & 26 & 20 &  0 & 19 & 13
\end{pmatrix}$$
with respect to the basis $H_3 H_3', \ldots, H_{10} H_{10}'$ of $V_{C, \eta}$ and the standard (monomial) basis of $S^2 \HH^0(C, \Omega_C(\eta))$. 

(6): We now compute the kernel of $\varphi$ which (as expected) is generated by the quadratic form $Q$ given by $11xw+yz$
that we started with.

(7): 
Computing the kernel of $(H^\square)^t$ we get a 3-dimensional vector space. The first two vectors correspond to the matrices:
$$Q_1 = \begin{pmatrix}1 & 0 & 16 & 29\\
     0 &  0 & 14 & 35\\
    16 & 14 & 5 &  8\\
    29 & 35 & 8 & 17
\end{pmatrix}, Q_2 = \begin{pmatrix}0 & 19 & 23 & 22\\
    19 & 0 & 17 & 6\\
    23 & 17 & 9 & 18\\
    22 & 6 & 18 & 2
\end{pmatrix}$$
and we get the symetric matrix
$$Q' = (Q_1^{-1}+Q_2^{-1})^{-1}= \begin{pmatrix}16 & 17 & 3 & 10 \\
17 & 31 & 35 & 17\\
 3 & 35 & 17 & 25 \\
10 & 17 & 25 & 26 \\
\end{pmatrix}.$$

(8): We now use this to compute the matrix representation of $\psi$,

$$M_{\psi} = \begin{pmatrix}21 & 33 & 35 & 22 & 4 & 35 & 23 \\
 9 & 29 & 27 & 6 & 0 & 16 & 31 \\
35 & 35 & 26 & 17 & 35 & 29 & 26 \\
 7 & 6 & 10 & 3 1& 13 & 36 & 20 \\
28 & 32 & 17 & 17 & 6 & 0 & 35 \\
13 & 6 & 1 & 15 & 29 & 2 & 16 \\
35 & 7 & 32 & 3 & 24 & 30 & 26
\end{pmatrix}$$
(9): Finally, we reconstruct the cubic curve as
\begin{align*}\sqrt{\det(\Delta)} = &22x^3 + 29x^2y + 13x^2z + 35x^2w + 25xy^2 + 27xyz + 18xyw + 5xz^2 +\\ &26xzw + 23xw^2 + 28y^3 + 22y^2z + 22y^2w + 24yz^2 + 5yzw + 4yw^2 +\\
&20z^3 + 12z^2w +  6zw^2 + 24w^3
\end{align*}
where 
    $\Delta := \begin{pmatrix}
d_1 & d_2 & d_3\\
d_2 & d_4 & d_5\\
d_3 & d_5 & d_6
\end{pmatrix}$  for $d := M_\psi \begin{pmatrix}
    H_1 H_1' & \ldots & H_7 H_7'
\end{pmatrix}^t.$
\end{example}

We now apply \Cref{thm:main_result} to two examples of arithmetic interest, demonstrating how our formulas can be employed in practical settings. See the files \href{https://github.com/JHanselman/reconstructing-g4/blob/main/examples/paper-examples/Example-4.1.m}{\texttt{Example-Gluing.m}} and \href{https://github.com/JHanselman/reconstructing-g4/blob/main/examples/paper-examples/Example-4.2.m}{\texttt{Example-Modular.m}} in the directory \texttt{examples/paper-examples} of \cite{GitHub} for the \Magma{} code used to compute these examples.
We computed these examples using Magma V2.29-2 on a Macbook running on MacOS Tahoe 26.4.1 with an Apple M2 Chip and 16GB of RAM.

\begin{example} \label{ex:gluing}
    (Gluing)
    Consider the hyperelliptic genus 2 curve
    $$
    C_1: y^2 = 24 x^5 + 36 x^4 - 4 x^3 - 12 x^2 + 1
    $$
    which has LMFDB label \href{https://www.lmfdb.org/Genus2Curve/Q/20736/l/373248/1}{\texttt{20736.l.373248.1}}.
    From the information on the curve's homepage, we see that the geometric endomorphism algebra of $\Jac(C_1)$ is the quaternion algebra $B_{2,3}$, i.e., the unique quaternion algebra over $\QQ$ ramified at $2$ and $3$. We look for interesting genus two curves $C_2$ such that the quotient of $A=\Jac(C_1)\times \Jac(C_2)$ by a maximal isotropic subgroup of $A[2]$ is the Jacobian of a smooth genus $4$ curve $C$ (possibly up to a quadratic twist). Using the criterion of \cite[Theorem 1.2]{BruinKulkarni} one finds that, for example, the curve    
    $$
    C_2: y^2 = 3 x^5 - 68 x^4 + 159 x^3 + 232 x^2 - 132 x + 16
    $$
    has this property. Using the methods of Costa--Mascot--Sijsling--Voight \cite{Endomorphism}, we verify that $\End^{0}_{\overline{\mathbb{Q}}}(\Jac(C_2))\cong \mathbb{Q}(\sqrt{5})$.
    \par Our formulas produce the following equations for the genus 4 curve $C$: 
    \begin{align*}
        0 &= 10x^2 + 8xy - 9y^2 - 33z^2 - 30zw - 40w^2\\
        0 &= -6x^3 - 2x^2y - xy^2 + 5xz^2 - 22xzw - 5xw^2 + 3y^3 + 11yz^2\\
            &\qquad + 10yzw - 11yw^2
 \, .
        \end{align*}
    By construction we have that $\End_{\overline{\QQ}}^0(\Jac(C))\cong B_{2,3} \times \mathbb{Q}(\sqrt{5})$, which we verified using the heuristic methods in \cite{Endomorphism}. As an additional check, we verified that the L-polynomial of $C$ is equal to the product to the L-polynomials of $C_1$ and $C_2$, up to twist, for all primes of good reduction up to $1000$.
    \par A priori the curve will come out in a coordinate system where the equations are not defined over $\mathbb{Q}$. Using the knowledge of a big period matrix for $\Jac(C)$ associated to a $\mathbb{Q}$-rational basis for $\HH^0(C, \Omega_C)$, we found the $\text{PGL}_4$ transformation that changes coordinates into this $\mathbb{Q}$-rational basis.

    Computing the period matrices to precision $200$, the computation took 1.18 seconds; the time taken by the various steps is given below.


   \begin{table}[h]
    \centering
    \begin{tabular}{|l|c|}
        \hline
        \textbf{Step} & \textbf{Time (seconds)} \\ \hline
        Theta constants & 0.15 \\ \hline
        Reconstruction over $\mathbb{C}$ & 0.47 \\ \hline
        Change of coordinates & 0.56 \\ \hline
        \textbf{Total time} & \textbf{1.18} \\ \hline
    \end{tabular}
    \label{tab:gluing_computation_times}
\end{table} 
\end{example}

\begin{example} \label{ex:modular}

    (Modular Jacobians) Let $f$ be the modular form orbit with LMFDB label \href{https://www.lmfdb.org/ModularForm/GL2/Q/holomorphic/778/2/a/a/}{778.2.a.a}. We use our method to compute the abelian fourfold that corresponds to it via modularity for RM abelian varieties over $\QQ$. This abelian variety $A$ is the subvariety of $J_0(778) =\Jac(X_0(778))$ corresponding to a 4-dimensional Hecke-invariant subspace. It has RM by the field of Hecke eigenvalues which is the totally real subfield of $\mathbb{Q}(\zeta_{15})$; this is the quartic field with LMFDB label \href{https://www.lmfdb.org/NumberField/4.4.1125.1}{4.4.1125.1}.
    We begin by computing its period matrix using \Magma{}'s command \texttt{Periods}. However, this period matrix need not correspond to a principally polarized abelian variety: the pullback of the principal polarization on $J_0(778)$ is not in general principal.
    \par To remedy this, we call the command \texttt{SomePrincipalPolarizations} from the \href{https://github.com/AndrewVSutherland/ModularCurves/}{ModularCurves} GitHub repository \cite{ModularCurves}. This produces several big period matrix candidates. Evaluating the Schottky modular form on the first of these yields a complex number with absolute value $6.2124\cdot10^{-300}$ when calculated with precision 300, indicating that this principally polarized abelian variety is likely a Jacobian. Applying our method then produces the corresponding genus 4 curve, which, after a change of variable, is isomorphic to the curve $C$ in $\PP^3$ given by the following equations.
    \begin{align*}
    0 &= x^2 - xz - xw - y^2 - yw + 2z^2 + zw - 4w^2\\
    0 &= 2xyw - 2x z^2 - 12 x z w - 10 x w^2 - y^2 z - 2 y^2 w + y z w + 4 y w^2 + 2 z^3\\
    &\qquad- 20 z w^2 - 18 w^3
    \end{align*}
    As a sanity check for this heuristic example, we compute the places of bad reduction of $C$. Using the methods from \cite{ThomasInvariants}, we compute the invariants of $C$ and in particular the discriminant. The result is $2^9\cdot 113^{30}\cdot 389^4$ and therefore candidates for the primes of bad reduction are $2, 113$, and $389$. However, modulo 113 the given equations reduce to a smooth genus $4$ curve, although the canonical quadric degenerates to a cone. Thus the primes of bad reduction are $2$ and $389$ as one would expect, since the level of $f$ is $778= 2\cdot 389$.
    As a further sanity check, we have computed and compared the local $L$-factors of $f$ and $C$ for primes up to 1000 and verified that they match. A priori the Jacobian of $C$ could have been a quadratic twist of the abelian variety corresponding to $f$, but our computations also showed this is not the case.

    Computing the period matrices to precision $300$, the computation took 1.51 seconds; the time taken by the various steps is given below. 


    \begin{table}[h]
        \centering
        \begin{tabular}{|l|c|}
            \hline
            \textbf{Step} & \textbf{Time (seconds)} \\ \hline
            Theta constants & 0.19 \\ \hline
            Reconstruction over $\mathbb{C}$ & 0.51 \\ \hline
            Change of coordinates & 0.81 \\ \hline
            \textbf{Total time} & \textbf{1.51} \\ \hline
        \end{tabular}
        \label{tab:computation_times_modular}
    \end{table}    

    We thank Noam Elkies for suggesting this example, as well as Edgar Costa for his help in computing polarizations.
\end{example}

\section{Higher genus} \label{sec:future}

We will discuss the opportunities and challenges that arise in higher genera. In work in progress 
\cite{HPS2} we plan to generalize the Milne correspondence to arbitrary $g$. This would give a bijection between pairs of odd theta hyperplanes in $\mathbb{P}^{g-1}$ for $C$ and odd theta hyperplanes in $\mathbb{P}^{g-2}$ for the Prym variety $\Prym(C, \eta)$. Notice that in this context the Prym variety is no longer automatically a Jacobian. Nevertheless, one can define odd theta hyperplanes for arbitrary p.p.a.v.s; however, these will lack the geometric interpretation as multitangents of a canonically embedded curve (unless the p.p.a.v. happens to be a Jacobian).
\par This generalization would give formulas for the equations of a generic genus $5$ curve in terms of its theta constants. Indeed, for $g\geqslant 5$ the image of the canonical embedding of a generic curve of genus $g$ is cut out by quadratic equations. For $g=5$, one should be able to compute these equations with a suitable adaptation of the results in Section \ref{subsec:quad}.
\par For $g=6$ we also expect that our method works. However, in order to get a formula purely in terms of theta constants, it remains to find a formula for the Jacobian nullvalues as a rational function in the theta constants; see also Remark \ref{rem:ConjIgusa}.
\par Here we use that in the cartesian diagram
\begin{equation}
\begin{tikzcd}
    S^2 \HH^0(\Omega_C (\eta))\times_{\HH^0(\Omega_C^{\otimes 2})} S^2 \HH^0(\Omega_C)\arrow{r}\arrow[hookrightarrow]{d}\arrow{r}{} & S^2 \HH^0(\Omega_C (\eta))\arrow[hookrightarrow]{d}\\
     S^2 \HH^0(\Omega_C)\arrow{r} & \HH^0(\Omega_C^{\otimes 2})    
\end{tikzcd}
\end{equation}
the vertical maps are injective for a generic curve of genus $g=5, 6$ which was proven by Donagi \cite[Theorem 5.1]{Donagi} for $g=5$ (here one interprets the right vertical map as the codifferential of the Prym map) and Beauville \cite[Proposition 7.10]{BeauvillePrym} for $g=6$.
\par For $g\geqslant 7$ it is not longer true that 
$$S^2 \HH^0(\Omega_C(\eta))\rightarrow \HH^0(\Omega_C^{\otimes 2}) $$
is injective (for simple dimension reasons). This means that we can no longer view $V_{C, \eta}$ as a linear subspace of $S^2\HH^0(C, \Omega_C)$ because here the injectivity of the map 
$$S^2\HH^0(C, \Omega_C(\eta))\rightarrow \HH^0(C, \Omega_C^{\otimes 2})$$
in (\ref{eqn:phi_diagram}) was crucially used. Nonetheless, one could still use odd theta data to produce a system of equations satisfied by the analogues of the unknown constants $\lambda_i$ from Equation (\ref{eqn:Milne10}), but these equations will be quadratic instead of linear. Therefore another idea is needed here. We plan to revisit this question in the future.
\begin{rem}
Lehavi's method \cite{Lehavi} generalizes to genus $5$ (see \cite{Lehavi5}) but not to genus $6$ because the fact that $m_{\eta}$ is an isomorphism causes problems. This is not a restriction for the method of the present paper.
\par In the recent preprint \cite{CelikLehavi}, Lehavi's technique is pushed to genus $6$ and $7$; their approach is based on equations of degree $4$.
\end{rem}

\appendix
\bibliographystyle{alphaurl}
\bibliography{bibl.bib}
\end{document}